\documentclass{ws-ijm}
\usepackage{
amsmath,
amssymb,
amsfonts,
amscd,eucal
}
\usepackage[all]{xy}
\usepackage{cite}
\usepackage{graphicx}
\usepackage{epsfig}
\usepackage{typearea}			
\areaset{17cm}{23cm}			

\usepackage{tikz}
\usetikzlibrary{lindenmayersystems}

\def\ca{{\mathcal A}}
\def\cb{{\mathcal B}}
\def\cc{{\mathcal C}}

\def\ch{{\mathcal H}}

\def\cl{{\mathcal L}}
\def\cam{{\mathcal M}}

\def\car{{\mathcal R}}

\def\ct{{\mathcal T}}
\def\cu{{\mathcal U}}

\def\bc{{\mathbb C}}

\def\bi{{\mathbb I}}

\def\bn{{\mathbb N}}
\def\br{{\mathbb R}}
\def\bt{{\mathbb T}}
\def\bz{{\mathbb Z}}
\def\bno{\bn\cup\{0\}}

\def\a{\alpha}
\def\b{\beta}
\def\g{\gamma}        \def\G{\Gamma}
\def\d{\delta}        
\def\eps{\varepsilon}

\def\l{\lambda}       
\def\m{\mu}

\def\r{\rho}
\def\s{\sigma}        
     
\def\t{\tau}

\def\f{\varphi}
\def\O{\Omega}

\DeclareMathOperator{\tr}{tr}
\DeclareMathOperator{\id}{id}
\DeclareMathOperator{\dom}{Dom}

\def\zci{\bm{i}}

\newcommand{\set}[1]{\left\{#1\right\}}

\newcommand{\ov}{\overline}
\newcommand{\wt}{\widetilde}
\newcommand{\wh}{\widehat}
\DeclareMathOperator{\ssv}{span}

\DeclareMathOperator{\diag}{diag}
\DeclareMathOperator{\End}{End}
\DeclareMathOperator{\Aut}{Aut}

\newcommand{\norm}[1]{\left\Vert#1\right\Vert}

\newcommand{\ad}{\textrm{ad}}
\newcommand{\Ad}{\textrm{Ad}}

\def\length{\mathrm{length}}
\def\size{\mathrm{size}}
\def\level{\mathrm{level}}

\def\dd{\d} 
\def\uu{U} 

\newcommand{\betaC}{\beta} 
\newcommand{\gammaCA}{\gamma} 

\newcommand{\piA}{\pi} 
\newcommand{\piAinf}{\psi} 
\newcommand{\ppinf}{\pi_\infty} 
\newcommand{\ppuniv}{\pi_u} 
\newcommand{\ppcov}{\wt{\pi_u}} 
\newcommand{\rrinf}{\rho_\infty} 
\newcommand{\rruniv}{\rho_u} 
\newcommand{\rrcov}{\wt{\rho_u}} 
\newcommand{\ssinf}{\sigma_\infty} 
\newcommand{\ssuniv}{\sigma_u} 
\newcommand{\sscov}{\wt{\sigma_u}} 
\def\repr{\psi} 
\newcommand{\ppW}{\ppinf \rtimes W_\infty} 
\newcommand{\ssW}{\ssinf \rtimes W_\infty} 
\newcommand{\ssU}{\sscov \rtimes U} 

\def\molt{r} 
\def\som{d} 

\begin{document}
\markboth{V.Aiello, D.Guido, T.Isola}{Spectral triples on irreversible $C^*$-dynamical systems}

\title{Spectral triples on irreversible $C^*$-dynamical systems}
\author{Valeriano Aiello}
\address{Mathematisches Institut, Universit\"at Bern, Alpeneggstrasse 22, 3012 Bern, Switzerland
\\
valerianoaiello@gmail.com}
\author{Daniele Guido}
\address{ 
Dipartimento di Matematica, Universit\`a di Roma ``Tor Vergata'', I--00133 Roma, Italy \\ guido@mat.uniroma2.it}
\author{Tommaso Isola}
\address{
Dipartimento di Matematica, Universit\`a di Roma ``Tor Vergata'', I--00133 Roma, Italy \\ isola@mat.uniroma2.it}

\maketitle

\begin{abstract}
Given a spectral triple on a $C^*$-algebra $\mathcal A$ together with a unital injective endomorphism $\alpha$, the problem of defining a suitable crossed product $C^*$-algebra endowed with a spectral triple is addressed.  The proposed construction is mainly based on the works of Cuntz and \cite{Skalski}, and on our previous papers \cite{AiGuIs01,AGI3}.
The embedding of $\alpha(\mathcal A)$ in $\mathcal A$ can be considered as the dual form of a covering projection between noncommutative spaces. A main assumption is the expansiveness of the endomorphism, which takes the form of the local isometricity of the covering projection, and is expressed via the compatibility of the Lip-norms on $\mathcal A$ and $\alpha(\mathcal A)$.
\end{abstract}



\section{Introduction}


How to promote a spectral triple on an algebra to a spectral triple on a crossed product $C^*$-algebra has been the subject of various papers \cite{Skalski,Paterson,BMR,GaGr,IoMa};   the same has been recently done  for the structure of compact quantum metric spaces \cite{KaKy}.

The aim of this paper is to tackle the following question: is it possibile to extend the construction of a spectral triple on a crossed product $C^*$-algebra based on a spectral triple on the base algebra to the case of crossed products with a single endomorphism?

Even though we do not have yet a general answer to this problem, we are able to propose a procedure - some steps of which can be completely described, while for others  we can give several examples - which explains what we expect to be the general case.

Before describing our plan, we draw attention to a feature of our construction, namely we more or less explicitly assume that our endomorphism  is in a sense expansive, a notion which  has been often considered both in the commutative and in the noncommutative case, see e.g. \cite{DGMW}.  Such property has important consequences: the compact resolvent property for the Dirac operator forces the spectral triple on the crossed product to be semifinite, and the bounded commutator property requires a reduction of the crossed product $C^*$-algebra, namely a new definition of crossed product by an endomorphism. 

Indeed, even though there are now various notions of crossed product of a $C^*$-algebra with an endomorphism, see e.g. \cite{Murphy,Exel,KwLe}, we essentially follow a path outlined by Cuntz \cite{Cuntz} and then further developed by Stacey \cite{StaceyCrossed}, but we are forced to adapt it to the case of expansive endomorphisms.

According to Cuntz, given a $C^*$-algebra $\ca$ together with a unital injective endomorphism $\alpha$, one constructs  a direct system of $C^*$-algebras $\ca_n$ with endomorphisms $\alpha_n$, whence the direct limit $C^*$-algebra $\ca_\infty$ is obtained. The key point is that the endomorphism $\a$ of $\ca$ becomes an automorphism $\a_\infty$ on $\ca_\infty$, so that one may define the crossed product $\ca\rtimes_\a\bn$ as  the crossed product $\ca_\infty\rtimes_{\alpha_\infty}\bz$. Let us note that this definition gives back the original algebra when $\alpha$ is an automorphism.

The first and second  step of our construction  have been studied in \cite{AiGuIs01,AGI3}, where one assumes that a spectral triple $\ct$ on $\ca$ is given.
Let us observe that unital injective endomorphisms of a $C^*$-algebra $\ca$ can be seen as noncommutative self-coverings of the underlying noncommutative space; the first step is then to endow any of the $C^*$-algebras $\ca_n$ described above with a spectral triple $\ct_n$ which makes the self-covering locally isometric or, equivalently, such that the Lip-norms induced by the Dirac operators are compatible with the connecting maps (this property can and will be weakened in some cases, cf. Section \ref{UHF}). This means that the sequence of covering spaces consists of dilated copies of the original space. This request is the reason of the expansivity mentioned above. Even if we do not give a general procedure for this step, this is  not a difficult task in all the examples considered in  \cite{AiGuIs01,AGI3}.

The second step consists of constructing a spectral triple $\ct_\infty$ on the direct limit $\ca_\infty$ which is in some sense naturally associated with the original spectral triple on $\ca$. We note here that the algebra $\ca_\infty$ can be naturally seen as the solenoid algebra associated with the pair ($\ca,\alpha$), see \cite{AiGuIs01,AGI3,DGMW,LP2} for related constructions. In the abelian case, an intrinsic notion of solenoid, called compact universal cover, has been studied in \cite{Plaut} in great generality.

Coming back to $\ct_\infty$, we wish to define it as a suitable limit of the triples $\ct_n$ on $\ca_n$. This step is far from being obvious, firstly because there is no general procedure to define a limit of a sequence of spectral triples (however in some circumstances one may follow  \cite{FloGho}), secondly because the situations we consider are quite different, ranging from regular coverings associated with an action of an abelian group to (possibly ramified) coverings with trivial group of deck transformations. Examples illustrating this step are contained in \cite{AiGuIs01,AGI3} and briefly described below. In all cases, the coverings becoming wider and wider, the spectra of the Dirac operators turn more and more closely packed, so that the limit has no longer compact resolvent. However, a corresponding rescaling of the traces gives rise to a (semicontinuous semifinite) trace on a suitable $C^*$-algebra $\cb$ of geometric operators, which contains $\ca_\infty$ and the resolvents of the limiting Dirac operator, finally producing a semifinite spectral triple on $\ca_\infty$. This means in particular that the semifiniteness property is true already at the level of $\ca_\infty$, therefore determines the analogous semifiniteness property for the spectral triple on the crossed product.

The third and final step, which is the main object of this paper, consists in defining a new kind of crossed product of a $C^*$-algebra w.r.t. an endomorphism, which can be seen as a variant of the crossed product considered by Cuntz in \cite{Cuntz} and Stacey in \cite{StaceyCrossed}, and which turns out to be tailored to accomodate a spectral triple in the case of expansive endomorphisms. 

The notion of this new crossed product with an endomorphism is given in Definition \ref{1.1}. On the one hand it is a universal object, therefore defines a unique object up to isomorphisms, on the other hand, as shown in Theorem \ref{prop:crossedProd}, it coincides with a reduction by a projection of the $C^*$-algebra crossed product defined in \cite{StaceyCrossed}, Proposition 1.13. While the latter is nothing else than the crossed product of $\ca_\infty$ with $\bz$ w.r.t. $\alpha_\infty$, our notion can be considered as the crossed product of $\ca_\infty$ with $\bn$ w.r.t. $\alpha_\infty$. 

Indeed, while Stacey crossed product with an endomorphism reduces to the usual crossed product for an automorphism $\alpha$, ours produces a ``corner'' of it, in such a way that only positive powers of $\alpha$ are implemented.

The advantage of such a choice is to allow the weakening of the request of metric equicontinuity (Lip-boundedness in our paper) of \cite{Skalski}, which, for an action $\alpha$ of $\bz$ and a Lipschitz element  $a$ reads $\displaystyle \sup_{n\in\bz} L(\alpha^{-n}(a)) < \infty$ and  makes sense for automorphisms, to a condition on $\alpha$ that we call Lip-semiboundedness, namely $\displaystyle \sup_{n\in\bn} L(\alpha^{-n}(a)) < \infty$.
More precisely, in Section \ref{subsec:CrossedProd}, we first generalize the construction of a spectral triple on a crossed product described  in \cite{Skalski} to the case of a semifinite spectral triple, maintaining the Lip-boundedness assumption, and then modify it by replacing the crossed product of Cuntz-Stacey with our crossed product, and noting that in this case the request of the endomorphism being Lip-semibounded
is sufficient to guarantee the bounded commutator property of the spectral triple, cf. Theorem \ref{triple-cross-prod-N}. Moreover, such theorem shows that the metric dimension of the crossed product spectral triple equals the metric dimension of the base triple increased by 1.
On the same grounds, a theory for the action of suitable semigroups (e.g. $\bn^k$) can be established, but this will not be discussed here.

In the last section of this paper we show that the self-coverings considered in \cite{AiGuIs01,AGI3} satisfy the Lip-semiboundedness  condition, hence give rise to a semifinite spectral triple on the crossed product. 

The first example deals with the self-covering of a $p$-torus, which is a regular covering. Given a purely expanding integer valued matrix $B$, the covering projection goes from $\br^p/B\bz^p$ to $\br^p/\bz^p$ and the canonical Dirac operator on the covering makes the covering projection locally isometric. A natural embedding of the $C^*$-algebra $\ca_n$ in $\cb(\ch_0)\otimes M_{{\molt}^n}(\bc)$ gives rise to the embedding of the direct limit $C^*$-algebra $\ca_\infty$ in $\cb(\ch_0)\otimes \mathrm{UHF}_{\molt}$, which is the algebra $\cb$ mentioned above, where $r=|\det(B)|$ 
and $\mathrm{UHF}_r$ denotes the infinite tensor product of $M_r(\bc)$. 
Moreover,  the Dirac operators $D_n$ converge in the norm resolvent sense to a Dirac operator affiliated with $\cb(\ch_0)\otimes \mathrm{UHF}_{\molt}$. This structure produces a semifinite spectral triple on $\ca_\infty$, as shown in \cite{AiGuIs01}. Theorem \ref{teo-p-toro} shows that the condition of Lip-semiboundedness is satisfied, hence we get a semifinite spectral triple on our crossed product with $\bn$.

The second example treats the case of regular noncommutative coverings of the rational rotation algebra with abelian group of deck transformations as defined in \cite{AiGuIs01}. The procedure and the results are essentially the same as the previous example, but the condition $r \equiv_q \pm 1$ has to be further assumed in order to get a self-covering. 

The third example concerns the UHF-algebra with the covering map given by  the shift endomorphism and  the spectral triple described in \cite{Chris}. In this case the Lip-norms given by the spectral triples are not compatible, namely $\|[D_n,\alpha^n(a)]\|\ne\|[D_0,a]\|$ for $a$ Lipschitz in $\ca_0$, however $\|[D_{n+p},\alpha^p(a)]\|$ is bounded in $p$ (indeed converges) for any Lipschitz element in $\ca_n$. Again we show  that the condition of Lip-semiboundedness is satisfied, cf. Theorem \ref{UHFcrossedprod}.

The fourth and last example describes the crossed product associated with a ramified covering of the fractal called Sierpi\'nski gasket. Such covering is not given by an action of a group of deck transformations. Here the spectral triple  on $\ca$ is the one described in \cite{GuIs16}, and the spectral triples on $\ca_n$ make the covering maps locally isometric. The $C^*$-algebra $\cb$ containing both $\ca_\infty$ and the resolvents of $D_\infty$ is an algebra of geometric operators acting on the $\ell^2$ space on the edges of the infinite Sierpi\'nski gasket with one boundary point \cite{Tep}. The proof of the condition of Lip-semiboundedness  is contained in Theorem \ref{teo-gasket}.

In all cases, by Theorem \ref{triple-cross-prod-N}, the spectral triples are finitely summable and their metric dimension is equal to the metric dimension of $\ct$  plus 1, namely it is the sum of the metric dimension of $\ct$ and the growth of $\bn$.

Finally, we mention that even though in all of our examples the functional given by the norm of the commutator with the Dirac operator is a Lip-norm in the sense of Rieffel \cite{Rieffel99} on $\ca$, such property does not hold for the spectral triple on the crossed product. In fact any distance on the state space of a unital $C^*$-algebra inducing the weak$^*$-topology should necessarily be bounded, and this is not the case for our construction. The reason is that the expansiveness of the endomorphism $\alpha$ produces larger and larger (quantum) covering spaces and eventually an unbounded solenoid space. This property leads to an analogous unboundedness for the distance on the state space of the crossed product $C^*$-algebra.
\section{Crossed products for $C^*$-algebras}

\subsection{Preliminaries}

\textbf{Inductive limit}. We begin by recalling the construction of the inductive limit $C^*$-algebra, due to Takeda \cite{Takeda}, for the particular case of interest in this paper, to fix some notation. Let $\ca$ be a unital $C^*$-algebra,  $\a\in \End(A)$ an injective, unital $*$-endomorphism. Consider the following inductive system
\begin{equation} \label{eq:CstarIndLim1}
\begin{CD}
	 A_0 @ > \f_0 >>    A_1 @ > \f_1 >> \cdots
\end{CD}
\end{equation}
where, for all $n\in\bn=\{0,1,2,\ldots\}$, $A_n=\ca$, $\f_n=\a$, and define, for $m< n$, $\f_{nm} : A_m \to A_n$ by $\f_{nm} := \f_{n-1}\circ \cdots \circ \f_m \equiv \a^{n-m}$, 
and $\f_{mm}:=\id$. Consider the direct product $\prod_{n=0}^\infty A_n$, with pointwise operations, and set 
$$
A_\infty  := \left\{ (a_n) \in \prod_{n=0}^\infty A_n : \exists m\in\bn \textrm{ such that } a_n=\f_{nm}(a_m)= \a^{n-m}(a_m), n\geq m \right\}/\!\sim \; ,$$ 
where $(a_n)\sim(b_n) \iff a_n=b_n$ for all large enough $n$. Then, $A_\infty$ is a $^*$-algebra. For all $n\in\bn$, define $\f_{\infty n}: a\in A_n\mapsto \f_{\infty n}(a)\in A_\infty$, where $\f_{\infty n}(a) \equiv (a_k)$, and
$$
a_k := 
\begin{cases}
	0, & k<n,\\
	\f_{kn}(a)=\a^{k-n}(a), & k\geq n.
\end{cases}
$$ 
We can introduce a norm $p$ on $A_\infty$ given by 
$$
p(a):= \limsup\limits_{n\to\infty} \Arrowvert \f_{nm}(a_m) \Arrowvert = \norm{a_m}\; ,
$$ 
if $a=\f_{\infty m}(a_m)$, which is independent of the representative, and is a $C^*$-norm. Upon completion, we get the desired inductive limit $C^*$-algebra, which is denoted $\ca_\infty \equiv \varinjlim A_n$.

\smallskip

\textbf{Crossed product}. Let us recall  the definition of the crossed product by an automorphism, in  the case of unital $C^*$-algebras, to fix some notation.

Let $\ca$ be a unital $C^*$-algebra, $\a \in \Aut(\ca)$ an automorphism. Denote by $C_c(\ca,\bz,\a)$ the $^*$-algebra of functions $f: \bz \to \ca$ with finite support, pointwise addition and scalar multiplication, with product $(fg)(n):= \sum_{k\in\bz} f(k) \a^k(g(n-k))$,  and involution $f^*(n):= \a^n(f(-n)^*)$, $f,g\in C_c(\ca,\bz,\a)$, $n\in\bz$. Define a norm on $C_c(\ca,\bz,\a)$ by $\Arrowvert f \Arrowvert_1 := \sum_{n\in\bz} \Arrowvert f(n)\Arrowvert$, and denote by $\ell^1(\ca,\bz,\a)$ the Banach $^*$-algebra obtained by completing $C_c(\ca,\bz,\a)$ with respect to this norm. A different description of $\ell^1(\ca,\bz,\a)$
is obtained by introducing the functions $\dd_n(k):= \d_{k, n}$. Then, $\ell^1(\ca,\bz,\a)$ is the set of all sums $\sum_{n\in\bz} a_n \dd_n$, with $a_n\in\ca$, for all $n\in\bz$, and $\sum_{n\in\bz} \Arrowvert a_n\Arrowvert <+\infty$. Let now $\pi$ be a representation of $\ca$ on $\ch$, $V$ a unitary operator on $\ch$, such that $\pi(\a(a)) = V \pi(a) V^*$, $a\in\ca$. The triple $(\ch, \pi,V)$ is called a covariant representation of $(\ca,\a)$. Then, the integrated form of $(\ch, \pi,V)$ is the representation $\pi\rtimes V$ of $C_c(\ca,\bz,\a)$ on $\ch$ given by 
\begin{eqnarray} \label{IntegratedForm}
	\pi\rtimes V(\sum_{n\in\bz} a_n\dd_n) &:=& \sum_{n\in\bz} \pi(a_n) V^n. 
\end{eqnarray}
It can be proved (\cite{Ped} Proposition 7.6.4) that there is a bijection between the set of non-degenerate covariant representations $(\ch, \pi,V)$ of $(\ca,\a)$ on a Hilbert space $\ch$, and the set of non-degenerate continuous representations of $\ell^1(\ca,\bz,\a)$ on $\ch$. Define the universal representation $\ppuniv$ of $\ell^1(\ca,\bz,\a)$ to be the direct sum of all non-degenerate continuous representations of $\ell^1(\ca,\bz,\a)$ on Hilbert spaces. The crossed product of $\ca$ by the action $\a$ of $\bz$ is the $C^*$-algebra $\ca\rtimes_\a\bz$ obtained as the norm closure of $\ppuniv(\ell^1(\ca,\bz,\a))$. 

\smallskip

\textbf{Reduced crossed product}. Since $\bz$ is an amenable group, a different description (\cite{Ped},  7.7.7) of the crossed product (called the reduced crossed product, in  the case of non amenable groups) can be given. Let $\pi$ be 
a faithful, non-degenerate representation
of $\ca$ on $\ch$, set $\wt\ch := \ell^2(\bz,\ch) \equiv \{ \xi : \bz\to\ch | \sum_{n\in\bz} \Arrowvert \xi(n) \Arrowvert^2 <+\infty \}$, and, for $n\in\bz$, $a\in\ca$, $\xi\in\wt\ch$,
\begin{align*}
	(\uu \xi)(n) := \xi(n-1), \qquad (\wt{\pi}(a)\xi )(n) := \pi( \a^{-n}(a) )(\xi(n)).
\end{align*}
Observe that, $\wt{\pi}(\a(a))= \uu\wt{\pi}(a) \uu^* $, $a\in \ca$.
Therefore, $(\wt\ch, \wt{\pi},\uu)$ is a covariant representation of $(\ca,\bz,\a)$, and the representation $\wt{\pi} \rtimes \uu$ is called a regular representation of $\ell^1(\ca,\bz,\a)$. In particular, if $a = \sum_{n\in\bz} a_n \dd_n \in C_c(\ca,\bz,\a)$,  then $(\wt{\pi} \rtimes \uu (a) \xi )(n) = \sum_{k\in\bz} \pi( \a^{-n}(a_k)) (\xi(n-k))$, $n\in\bz$.
Define the universal regular representation $\l_u$ of $\ell^1(\ca,\bz,\a)$ to be the direct sum of all regular representations of $\ell^1(\ca,\bz,\a)$ on Hilbert spaces. The (reduced) crossed product of $\ca$ by the action $\a$ of $\bz$ is the $C^*$-algebra obtained as the norm closure of $\l_u(\ell^1(\ca,\bz,\a))$. Observe that (\cite{Ped},  7.7.4), if $\pi_u$ is the universal representation of $\ca$, then $\ca\rtimes_\a\bz$ coincides with the norm closure of $\wt{\pi_u}\rtimes \uu (\ell^1(\ca,\bz,\a))$.
Therefore,  we get $\ca \rtimes_{\a} \bz = \langle\, \wt{\pi_u}(\ca),\uu \,\rangle$, where $\langle\, \wt{\pi_u}(\ca),\uu \,\rangle$ stands for the $C^*$-algebra generated by $\wt{\pi_u}(\ca)$ and $\uu$.

\smallskip

\textbf{Lift of a spectral triple to a crossed product}. 
First of all we recall the definition of spectral triples.

\begin{definition}\label{deftripla}
An odd spectral triple $(\cl,\ch,D)$ consists of a Hilbert space $\ch$,
an algebra  $\cl$ acting (faithfully) on it,
  a self adjoint operator  $D$  on the same Hilbert space such that $a\dom(D) \subset \dom(D)$ and $[D, a]$ is bounded for any $a\in\cl$, and with $D$ having compact resolvent.
A spectral triple is said to be even if there exists a  self-adjoint unitary operator $\G$ such that $\pi(a)\G = \G\pi(a)$, $\forall a\in\ca$, and $D\G=-\G D$. 
\end{definition}

In \cite{BMR}, Bellissard, Marcolli and Reihani show how to lift a spectral triple  from a unital C*-algebra $\ca$, endowed with an automorphism $\a$, to the crossed product $\ca\rtimes_\a \bz$. Their setting is generalised in (\cite{Skalski}, Theorem 2.8) to the case of the action of a discrete group. In the particular case of an automorphism, one obtains


\begin{definition}\label{equi}
Let $\ca$ be a unital C*-algebra, $\a\in \Aut(\ca)$ a unital automorphism, $(\cl,\ch,D)$ a spectral triple on $\ca$ such that $\a(\cl)\subset\cl$.
The automorphism is said to be Lip-bounded if
$$
\sup_{n\in\bz} \| [D,\a^{-n}(a)] \| < \infty, \qquad \forall a\in\cl.
$$
\end{definition}
The previous notion was introduced in  \cite{Skalski} where it is called the metric equicontinuity of the action.

\begin{theorem} 
Let $\ca$ be a unital C*-algebra, $(\cl,\ch,D)$ an odd spectral triple on $\ca$, and
$\a\in \Aut(\ca)$ a unital Lip-bounded automorphism. Set
\begin{flalign*}
& \cl_\rtimes := {}^*\mathrm{alg}(\wt{\pi_u}(\cl),U), & \qquad  & \ch_\rtimes   := \ch \otimes \ell^2(\bz) \otimes \bc^2, & \\
& D_\rtimes := D \otimes I \otimes \eps_1 + I \otimes D_\bz \otimes \eps_2, & \qquad  & \G_\rtimes  := I \otimes I \otimes \eps_3, & 
\end{flalign*}
where ${}^*\mathrm{alg}(\wt{\pi_u}(\cl),U)$ is the $^*$-algebra generated by $\wt{\pi_u}(\cl)$ and $U$, $(D_\bz \xi)(n) := n\xi(n)$, $\forall \xi\in\ell^2(\bz)$,
and 
\begin{align}\label{pauli}
\eps_1 := \begin{pmatrix}
0 & 1 \\
1 & 0
\end{pmatrix}, \; \eps_2 := \begin{pmatrix}
0 & -i \\
i & 0
\end{pmatrix}, \; \eps_3 := \begin{pmatrix}
1 & 0\\
0 & -1
\end{pmatrix}
\end{align}
are the Pauli matrices. %
\smallskip

Then $(\cl_\rtimes,\ch_\rtimes,D_\rtimes,\Gamma_\rtimes)$ is an even spectral triple on $\ca\rtimes_\a \bz$.
\end{theorem}

\begin{remark} 
In \cite{KaKy} a more general notion than Lip-boundeness of an automorphism is introduced, there called quasi-isometricity (see their Definition 4.4), and it is shown (in Example 4.5) that quasi-isometric automorphisms arise naturally in differential geometry.
\end{remark}


\subsection{A new definition of crossed product by an endomorphism}

There are many different definitions of the crossed product with an endomorphism, see e.g. \cite{Murphy}, \cite{Exel}, and the very general one given in \cite{KwLe}. We will work with a modification of the one introduced in \cite{Cuntz,StaceyCrossed}. Indeed, Cuntz (\cite{Cuntz}, pag. 101) considers the inductive sequence \eqref{eq:CstarIndLim1}, and its inductive limit C*-algebra $\ca_\infty$, which is endowed with an automorphism $\a_\infty$,  uniquely defined by the diagram \eqref{eq:CstarIndLim2}
\begin{equation} \label{eq:CstarIndLim2}
	\xymatrix{
	& \ca \ar[rr]^{ \a } \ar[dd]^{\a} &&  \ca \ar[rr]^{ \a  } \ar[dd]^{\a} && \ca \ar[rr]^{  \a  } \ar[dd]^{\a} && \cdots \ar[r]  &  \ca_\infty \ar[dd]^{\a_\infty} \\
	&&&&&&&&\\
	& \ca \ar[rr]^{  \a  } \ar[uurr]^{id} && \ca \ar[rr]^{ \a  } \ar[uurr]^{id} && \ca \ar[rr]^{ \a  } \ar[uurr]^{id} && \cdots \ar[r] & \ca_\infty
         }
\end{equation}
where the diagonal maps define the inverse $\a_\infty^{-1}$. Then Cuntz defined $\ca\rtimes_\a \bn := q(\ca_\infty \rtimes_{\a_\infty} \bz)q$, where $q\in\ca_\infty$ is the image of $1\in\ca$, and turns out to be $q=1$ in our case, since $\a$ is unital. Subsequently, Stacey \cite{StaceyCrossed} characterised $\ca\rtimes_\a \bn$ as the solution of a universal problem. 

In this paper, our interest is in lifting suitable spectral triples from $(\ca,\a)$, where $\a\in \End(\ca)$, to $\ca \rtimes_\a \bn$. Since we already know how to lift a spectral triple from $(\ca,\a)$ to $(\ca_\infty,\a_\infty)$, at least in some examples \cite{AiGuIs01,AGI3}, and the lift from $(\ca_\infty,\a_\infty)$ to $\ca_\infty \rtimes_{\a_\infty} \bz$ is well known \cite{BMR}, we found only natural to use Cuntz' definition of the crossed product $\ca\rtimes_\a \bn$. Unfortunately, the spectral triples $(\cl_\infty,\ch_\infty,D_\infty)$ on $(\ca_\infty,\a_\infty)$ we constructed in \cite{AiGuIs01,AGI3} satisfy, besides $\a_\infty(\cl_\infty)\subset \cl_\infty$, only $\sup_{n\in\bn} \| [D_\infty,\a_\infty^{-n}(a)] \| < \infty$, $\forall a\in\cl_\infty$. This fact forces us to introduce a modification in Cuntz' procedure, namely to consider $\ca \rtimes_\a\bn :=p(\ca_\infty \rtimes_{\a_\infty} \bz)p$, where $p \in \cb(\ell^2(\bz,\ch_u))$ is the projection on the non-negative ``frequencies'' 
\begin{equation} \label{eq:pProjection}
(p \xi)(n)  = 
\begin{cases}
	\xi(n), & n\geq 0,\\
	0, & n<0.
\end{cases}
\end{equation} 
Actually, we prefer to define our version of the crossed product by an endomorphism, in the same spirit of Stacey, as the solution to a universal problem, see Definition \ref{1.1}, and then prove in Theorem \ref{prop:crossedProd} that it coincides with $p(\ca_\infty \rtimes_{\a_\infty} \bz)p$.


\begin{definition}\label{Def:CovRep}
	Let $\ca$ be a unital $C^*$-algebra,  $\a\in \End(A)$ a $^*$-endomorphism.  Let  $\piA:\ca\to \cb(\ch)$ be a representation, $W\in \cb(\ch)$  an isometry. We say that $(\ch,\piA,W)$ is a covariant representation of  $(\ca,\a)$ on $\ch$, if 
\begin{align*}
	\piA(\a(a))W & = W\piA(a), \quad  a\in \ca,\\
	W^kW^{*k} & \in \piA(\ca)', \quad k\in\bn.
\end{align*}
\end{definition}


\begin{definition}  
\label{1.1}
	Let $\ca$ be a unital $C^*$-algebra,  $\a\in \End(A)$ an injective, unital $*$-endomorphism. The crossed product of $\ca$ with $\bn$ by $\a$ is a unital $C^*$-algebra $\cb$, together with a unital $^*$-monomorphism $\zci_\ca: \ca\to\cb$, and an isometry $t\in\cb$, such that
\begin{enumerate}
\item[$(1)$] $\cb$ is the $C^*$-algebra generated by $\zci_\ca(\ca)$ and $t$,

\item[$(2)$] $\zci_\ca(\a(a))t = t\zci_\ca(a)$, $a\in\ca$,

\item[$(3)$] $t^k(t^*)^k$ commutes with $\zci_\ca(\ca)$, $k\in\bn$,

\item[$(4)$] for every covariant representation $(\ch,\piA,W)$ of $(\ca,\a)$, there exists a non-degenerate representation $\widehat{\piA}$ of $\cb$ on $\ch$, such that $\widehat{\piA}\circ \zci_\ca = \piA$, and $\widehat{\piA}(t)=W$.
\end{enumerate}
\end{definition}


We denote by $\ca \rtimes_\a \bn$ the above algebra $\cb$. We have defined our crossed product as a universal object, which guarantees its uniqueness. For its existence, we will  prove in Proposition \ref{prop:crossedProd} that it is a reduction by a projection of the $C^*$-algebra crossed product defined by Cuntz in \cite{Cuntz}.

\subsection{Existence of the universal object}

Let us now consider the commutative diagram \eqref{eq:CstarIndLim2}.
It follows from (\cite{WO}, Theorem L.2.1) that the vertical maps determine a $^*$-homomorphism $\a_\infty:\ca_\infty\to\ca_\infty$, and the diagonal maps define the inverse of $\a_\infty$.

\begin{proposition} \label{prop:exists}
	Let $\ca$ be a unital $C^*$-algebra,  $\a$ a unital, injective $^*$-endomorphism of $\ca$. Then, there exists a covariant representation $(\ch,\piA,W)$ of $(\ca,\a)$. 
\end{proposition}
\begin{proof}
Let $\repr$ be a faithful representation of $\ca_\infty \rtimes_{\a_\infty} \bz$ on a Hilbert  space $H$. If $\pi_u$ is the universal representation of $\ca_\infty$, let $\ppcov : \ca_\infty \to  \ca_\infty \rtimes_{\a_\infty} \bz$, $\uu \in \cu( \ca_\infty \rtimes_{\a_\infty} \bz)$ be such that $\ca_\infty \rtimes_{\a_\infty} \bz = \langle \ppcov(\ca_\infty), \uu \rangle$, $\piA := \repr \circ \ppcov \circ \f_{\infty 0} : \ca \to \cb(H)$, which is a representation of $\ca$ on $H$, and $W:= \repr(\uu) \in \cb(H)$, which is a unitary operator acting on $H$. Moreover, for all $a\in\ca$, $k\in\bn$, by using that $\f_{\infty 0}\circ\a=\a_\infty \circ \f_{\infty 0}$ and $\ppcov(\a_\infty(x))=\uu\ppcov(x)\uu^*$, we get
\begin{align*}
	\piA(\a(a)) W & = (\repr\circ \ppcov \circ \f_{\infty 0}(\a(a)) \cdot \repr(\uu) ) = \repr( \ppcov \circ \a_\infty \circ \f_{\infty 0}(a) \cdot \uu ) \\
	& = \repr( \uu \cdot \ppcov \circ \f_{\infty 0}(a)) = (\repr(\uu))  (\repr\circ \ppcov \circ \f_{\infty 0}(a)) \\
	& = W\piA(a), \\
	W^kW^{*k} & = \repr(\uu^k \uu^{*k}) = 1 \in \piA(\ca)'.
\end{align*}
\end{proof}


We now prove that any covariant representation of $(\ca,\a)$ lifts to a covariant representation of $(\ca_\infty,\a_\infty)$.

\begin{proposition} \label{prop:RepIndLim}
	Let $\ca$ be a unital $C^*$-algebra,  $\a$ a unital, injective $^*$-endomorphism of $\ca$, and denote by $\ca_\infty$ the $C^*$-algebra inductive limit of the inductive system \eqref{eq:CstarIndLim1},
and denote by $\a_\infty$ the automorphism of $\ca_\infty$ induced by $\a$. 
Let $(\ch,\piA,W)$ be a covariant representation of $(A,\a)$, and denote by $\ch_\infty \equiv \varinjlim H_n$ the Hilbert space inductive limit of the inductive system 
\begin{equation} \label{HilbertIndLim}
\begin{CD}
	 H_0 @ > S_0 >>    H_1 @ > S_1 >> \cdots
\end{CD}
\end{equation}
where, for all $n\in\bn$,  $H_n:=\ch$, $S_n:= W$. Then, there exist $W_\infty\in\cu(\ch_\infty)$, and a covariant representation $(\ch_\infty, \ppinf, W_\infty)$ of $(\ca_\infty,\a_\infty)$, such that 
\begin{align*}
	\ppinf\circ \f_{\infty n}(a) S_{\infty n}  & = S_{\infty n} \piA(a), \quad n\in\bno, a\in\ca, \\
	W_\infty S_{\infty 0} & = S_{\infty 0} W,
\end{align*} 
where $S_{\infty n}: \xi\in H_n \mapsto (\xi_k)\in \ch_\infty$, $\xi_k:= 
\begin{cases}
	0, & k<n, \\
	W^{k-n}\xi, & k\geq n.
\end{cases}$
\end{proposition}
\begin{proof}
Denote by $W_\infty$ the unitary operator on the inductive limit $\ch_\infty \equiv \varinjlim H_n$ defined by the following diagram 
\begin{equation*}
	\xymatrix{
	& \ch \ar[rr]^{ W } \ar[dd]^{ W } &&  \ch \ar[rr]^{ W  } \ar[dd]^{ W } && \ch \ar[rr]^{ W  } \ar[dd]^{ W } && \cdots \ar[r]  &  \ch_\infty \ar[dd]^{ W_\infty } \\
	&&&&&&&&\\
	& \ch \ar[rr]^{ W } \ar[uurr]^{id} && \ch \ar[rr]^{ W  } \ar[uurr]^{id} && \ch \ar[rr]^{ W  } \ar[uurr]^{id} && \cdots \ar[r] & \ch_\infty
         }
\end{equation*}
so that $W_\infty S_{\infty n} = S_{\infty,n-1}$, for all $n\in\bn$, $n\geq 1$, and $W_\infty S_{\infty 0} = S_{\infty 0}W$. 

Introduce a map $\piAinf_0:\ca \to \cb(\ch_\infty)$ by 
\begin{align*}
	\piAinf_0(a) S_{\infty m}\xi := S_{\infty m} \piA(\a^m(a)) \xi, \quad a\in\ca, m\in\bn, \xi\in H_m \equiv \ch,
\end{align*}
which is well defined, because, if $S_{\infty m}\xi = S_{\infty,m-1}\eta = S_{\infty m}W\eta$, then $\xi=W\eta$, and 
\begin{align*}
	S_{\infty m} \piA(\a^m(a)) \xi & = S_{\infty m} \piA(\a^m(a)) W\eta = S_{\infty m} W\piA(\a^{m-1}(a))\eta = S_{\infty,m-1} \piA(\a^{m-1}(a))\eta.
\end{align*}
Let us prove that $\piAinf_0$ is a representation of $\ca$. Indeed, for $a,b\in\ca$, we get, for all $m\in\bn$, $\xi\in H_m$,
\begin{align*}
	\piAinf_0(ab) S_{\infty m}\xi & = S_{\infty m} \piA(\a^m(ab)) \xi = S_{\infty m} \piA(\a^m(a))  \piA(\a^m(b)) \xi \\
	& =   \piAinf_0(a) S_{\infty m}   \piA(\a^m(b)) \xi  = \piAinf_0(a) \piAinf_0(b) S_{\infty m}  \xi.
\end{align*}
Moreover, for $a\in\ca_\infty$, $\xi,\eta\in\ch$, $m,n\in\bz$, we get, if $n<m$,
\begin{align*}
	(S_{\infty m}\xi, \piAinf_0(a)^*S_{\infty n}\eta) & = (\piAinf_0(a)S_{\infty m}\xi, S_{\infty n}\eta) = (S_{\infty m}\piA(\a^m(a))\xi, S_{\infty n}\eta) \\
	& = (S_{\infty m}\piA(\a^m(a))\xi, S_{\infty m}S_{mn}\eta) = (\piA(\a^m(a))\xi, S_{mn}\eta) \\
	& = (\xi, \piA(\a^m(a^*))W^{m-n}\eta) = (\xi, W^{m-n}\piA(\a^n(a^*))\eta) \\
	& = (S_{\infty m}\xi, S_{\infty n}\piA(\a^n(a^*))\eta) = (S_{\infty m}\xi, \piAinf_0(a^*)S_{\infty n}\eta).
\end{align*}

Setting, for all $n\in\bn$, $\piAinf_n := \Ad(W_\infty^*)^n\circ\piAinf_0$, we get, for $m\geq n+1$, 
\begin{align*}
	\piAinf_{n+1}(\a(a)) S_{\infty m} & = (W_\infty^*)^{n+1}\piAinf_0(\a(a))W_\infty^{n+1} S_{\infty m} = (W_\infty^*)^{n+1} \piAinf_0(\a(a)) S_{\infty,m-n-1} \\
	& = (W_\infty^*)^{n+1} S_{\infty,m-n-1} \piA(\a^{m-n}(a)) = (W_\infty^*)^{n} S_{\infty,m-n} \piA(\a^{m-n}(a)) \\
	& = (W_\infty^*)^{n} \piAinf_0(a) S_{\infty,m-n} = (W_\infty^*)^{n} \piAinf_0(a) W_\infty^n S_{\infty m} = \piAinf_n(a) S_{\infty m},
\end{align*}
so that $\piAinf_{n+1}(\a(a)) = \piAinf_n(a)$. Therefore, the following diagram commutes
\begin{equation*} 
\begin{CD}
	 A_0 @ > \f_0 >>    A_1 @ > \f_1 >> A_2 @ > \f_2 >> \cdots @ >>> \ca_\infty \\
	 @V \piAinf_0 VV @V \piAinf_1 VV @V \piAinf_2 VV @. @V \ppinf VV \\
	 \cb(\ch_\infty) @ > \id >>  \cb(\ch_\infty) @> \id >> \cb(\ch_\infty) @> \id >> \cdots @>>> \cb(\ch_\infty)	 
\end{CD}
\end{equation*}
so that there is a unique $^*$-homomorphism $\ppinf:\ca_\infty\to\cb(\ch_\infty)$ such that $\ppinf\circ \f_{\infty n} = \piAinf_n$, for all $n\in\bn$. Therefore,  for all $n\in\bn$, $a\in\ca$, we have 
\begin{align}
	\ppinf\circ \f_{\infty n}(a) S_{\infty n}  & = \piAinf_n(a)S_{\infty n} = W_\infty^{*n}\piAinf_0(a) W_\infty^n S_{\infty n} = W_\infty^{*n}\piAinf_0(a)  S_{\infty 0} \notag \\
	& = W_\infty^{*n}  S_{\infty 0} \piA(a) = S_{\infty n} \piA(a). \label{allaccia}
\end{align}
Finally, for all  $n\in\bn$, $n\geq 1$, $a\in A_n = \ca$, we have 
\begin{align*}
	\ppinf\circ\a_\infty\circ\f_{\infty n}(a) & = \ppinf\circ\f_{\infty n}\circ\a(a) =  \piAinf_n\circ\a(a)  = \piAinf_{n-1}(a) \\ 
	& = \Ad(W_\infty) \circ \piAinf_n(a)  = \Ad(W_\infty) \circ \ppinf\circ \f_{\infty n} (a),
\end{align*}
so that $\ppinf\circ \a_\infty = \Ad(W_\infty) \circ \ppinf$, that is $(\ch_\infty,\ppinf,W_\infty)$ is a covariant representation of $(\ca_\infty,\a_\infty)$.
\end{proof}


\bigskip

We recall that in the construction of $\ca_\infty \rtimes_{\a_\infty} \bz$ we 
denoted by $\pi_u$ the universal representation of $\ca_\infty$ on $\ch_u$,  so that  $\ca_\infty \rtimes_{\a_\infty} \bz = \langle\, \ppcov(\ca_\infty),\uu \,\rangle$. 

\noindent Define the projection $p\in \cb( \ell^2(\bz,\ch_u) )$ as in \eqref{eq:pProjection}, so that $p\ppcov(a)=\ppcov(a)p$, $a\in\ca_\infty$, and set $t:= p\uu p \equiv \uu p$, so that $t^*t = p$,  and $t\ppcov(a) = \ppcov(\a_\infty(a))t$, $a\in \ca_\infty$. Set $\zci_\ca(a) := \ppcov\circ \f_{\infty 0}(a)p$, which is a representation of $\ca$ on $p\ell^2(\bz,\ch_u)$, and denote by $C^*(\ca,\a,\bn)$ the $C^*$-algebra generated by $\zci_\ca(\ca)$ and $t$ on $p\ell^2(\bz,\ch_u)$. 


\begin{proposition} \label{prop:generate}
For any $a\in\ca$, $k\in\bn$, we have that 
\begin{itemize}
\item[$(1)$] $\zci_\ca(\a(a))t  = t\zci_\ca(a)$, 

\item[$(2)$] $t^k(t^*)^k \zci_\ca(a) = \zci_\ca(a) t^k(t^*)^k$,

\item[$(3)$] $C^*(\ca,\a,\bn)  \equiv \langle\, \zci_\ca(\ca), t \,\rangle = \langle\, p\ppcov(\ca_\infty)p, p\uu p  \,\rangle = \langle\, t^{*m}\zci_\ca(a)t^n: a\in\ca, m,n\in\bn \,\rangle = \langle\, \zci_\ca(a)t^mt^{*n}: a\in\ca, m,n\in\bn \,\rangle$.
\end{itemize}
\end{proposition}
\begin{proof}
$(1)$ Indeed, for all $a\in\ca$, 
\begin{align*}
	\zci_\ca(\a(a))t & = \ppcov\circ\f_{\infty 0}\circ\a(a)t = \ppcov\circ\a_\infty\circ\f_{\infty 0}(a)t = t\ppcov\circ\f_{\infty 0}(a)p = t\zci_\ca(a).
\end{align*}

\noindent $(2)$ Indeed, since $t^k = \uu^kp = p\uu^kp$, we get
\begin{align*}
t^k(t^*)^k \zci_\ca(a) & = U^kp(U^*)^k \ppcov \circ\f_{\infty 0}(a)p = U^kp \ppcov \circ\a_\infty^{-k} \circ \f_{\infty 0}(a) (U^*)^kp  \\
& = U^k \ppcov \circ\a_\infty^{-k} \circ \f_{\infty 0}(a) p(U^*)^kp  = \ppcov  \circ \f_{\infty 0}(a) U^k  p(U^*)^kp  \\
& = \ppcov  \circ \f_{\infty 0}(a) pU^k  p(U^*)^kp  = \zci_\ca(a) t^k(t^*)^k.
\end{align*}

\noindent $(3)$ Indeed, $\ca_\infty = \ov{\ssv} \{ \f_{\infty m}(a) : a\in\ca, m\in\bn \}$, and
\begin{align*}
	p\ppcov\circ \f_{\infty m}(a)p  & = p\ppcov\circ \f_{\infty m}(a) \uu^{-m}\uu ^{m}p = p \uu^{-m} \ppcov \circ \a_\infty^m \circ\f_{\infty m}(a) \uu ^{m} p \\
	& = p \uu^{-m} \ppcov \circ \f_{\infty m} \circ \a^m(a) \uu ^{m }p  = p \uu^{-m} \ppcov \circ \f_{\infty 0}(a) \uu ^{m }p  = t^{*m} \zci_\ca(a) t^{m},
\end{align*}
so that $\langle\, p\ppcov(\ca_\infty)p, p\uu p  \,\rangle = \langle\, t^{*m}\zci_\ca(a)t^n: a\in\ca, m,n\in\bn \,\rangle = \langle\, \zci_\ca(a)t^mt^{*n}: a\in\ca, m,n\in\bn \,\rangle =  \langle\, \zci_\ca(\ca), t \,\rangle$.
\end{proof}


We want to prove that $C^*(\ca,\a,\bn)$ is isomorphic to the crossed product of $\ca$ with $\a$ by $\bn$. Actually, property $(1)$ in Definition \ref{1.1} follows by definition, while properties $(2)$ and $(3)$  have been proved in Proposition \ref{prop:generate}. Unfortunately, the proof of property $(4)$ in Definition \ref{1.1} will force us to a long detour. First of all, we need a $C^*$-algebra which contains $\ca_\infty\rtimes_{\a_\infty}\bz$ and a projection on the ``positive frequencies'' of $\bz$, and to which we can lift, in a canonical way, any representation of $\ca_\infty$. We start with some preliminary results. 
Denote by $\bz_\infty:=\bz\cup\{+\infty\}$ the spectrum of the $C^*$-algebra of  functions on $\bz$, vanishing at $-\infty$, and having finite limit for $n\to+\infty$, and let $\betaC$ be the automorphism of $C_0(\bz_\infty)$ given by $\betaC(f)(n):= f(n-1)$, $n\in\bz$. 

It follows from \cite{Sakai}, Proposition 1.22.3,  that $\ca_\infty \otimes C_0(\bz_\infty) \cong C_0(\bz_\infty,\ca_\infty)$, that is two-sided sequences of elements in $\ca_\infty$, vanishing at $-\infty$, and having norm-limit for $n\to+\infty$. It follows from \cite{Take},  Proposition IV.4.22 that there is a unique automorphism $\gammaCA\in \Aut( C_0(\bz_\infty;\ca_\infty) )$ such that $\gammaCA(a\otimes f) = \a_\infty(a)\otimes \betaC(f)$, $a\in\ca_\infty$, $f\in C_0(\bz_\infty)$.

In Proposition \ref{prop:reprOfCrossed}, we construct a  representation of $C_0(\bz_\infty;\ca_\infty) \rtimes_{\gammaCA} \bz$ on $\ell^2(\bz,\ch_u)$. Let $\rruniv$ be the representation of $C_0(\bz_\infty)$  on $\ch_u$ given by  $\rruniv(f)\xi = f(0)\xi$, $f\in C_0(\bz_\infty)$, $\xi\in\ch_u$. It follows from \cite{Take}, Proposition IV.4.7, that there is a unique representation $\ssuniv$ of $C_0(\bz_\infty;\ca_\infty)$ on $\ch_u$, such that $\ssuniv(a\otimes f) = \ppuniv(a)\rruniv(f)$, $a\in\ca_\infty$, $f\in C_0(\bz_\infty)$. 

Introduce the representations $\rrcov$ of $C_0(\bz_\infty)$ and $\sscov$ of $C_0(\bz_\infty;\ca_\infty)$ on $\ell^2(\bz,\ch_u)$ given by, for $a\in\ca_\infty$, $f\in C_0(\bz_\infty)$, $\xi \in \ell^2(\bz,\ch_u)$, $n\in\bz$,
\begin{align*}
	(\rrcov(f) \xi)(n)  & :=  \rruniv( \betaC^{-n}(f) ) \xi(n) = f(n)\xi(n), \\
	(\sscov(a\otimes f) \xi)(n)  & :=  \ssuniv( \gammaCA^{-n}(a\otimes f) ) \xi(n).
\end{align*}


\begin{proposition} \label{prop:reprOfCrossed}
\begin{itemize}

\item[$(1)$] $\uu \rrcov(f)\uu^* = \rrcov(\betaC(f))$, $f\in C_0(\bz_\infty)$.

\item[$(2)$]	The representation $\sscov$ of $C_0(\bz_\infty;\ca_\infty)$ on $\ell^2(\bz,\ch_u)$ is faithful, and
\begin{align*}
	\sscov( C_0(\bz_\infty;\ca_\infty) ) & =  \langle\, \ppcov(\ca_\infty), \rrcov(C_0(\bz_\infty)) \,\rangle, \\
	\uu \sscov(a\otimes f)\uu^* & = \sscov(\gammaCA(a\otimes f)), \quad a\in\ca_\infty, f\in C_0(\bz_\infty).
\end{align*}

\item[$(3)$] The regular representation $\chi:=\ssU$ of $C_0(\bz_\infty,\ca_\infty) \rtimes_{\gammaCA} \bz$, induced from $\ssuniv$ on $\ell^2(\bz,\ch_u)$, is faithful.
\end{itemize}

\end{proposition}
\begin{proof}
$(1)$  is a computation.

\noindent $(2)$ It is easy to see that $(\sscov(g)\xi)(k) = \pi_u(\a_\infty^{-k}(g(k)))\xi(k)$, $k\in\bz$, $\xi\in\ell^2(\bz,\ch_u)$, $g\in C_0(\bz_\infty,\ca_\infty)$, from which it follows that $\sscov$ is faithful. Moreover, for $a\in\ca_\infty$, $f\in C_0(\bz_\infty)$, one has
\begin{align*}
	\sscov(\gammaCA( a\otimes f)) & = \sscov(\a_\infty(a)\otimes \betaC(f)) = \ppcov(\a_\infty(a)) \rrcov(\betaC(f)) \\
	& = \uu \ppcov(a) \uu^* \uu \rrcov(f) \uu^* = \uu \sscov(a\otimes f) \uu^*.
\end{align*}

\noindent $(3)$ This follows from \cite{Will}, Theorem 7.13.
\end{proof}

It follows from the previous Proposition that $\cc:= \langle\, \ppcov(\ca_\infty), \rrcov(C_0(\bz_\infty)), \uu \,\rangle \subset \cb(\ell^2(\bz,\ch_u))$ is isomorphic, via $\chi^{-1}$, to $(\ca_\infty \otimes C_0(\bz_\infty)) \rtimes_{\gammaCA} \bz$, and contains $C^*(\ca,\a,\bn)$. It follows from its construction that we can lift canonically  to $\cc$ any representation of $\ca_\infty$, as we prove in Proposition \ref{prop:Sigma}.

We now begin the proof of property $(4)$ in Definition \ref{1.1}. In rough terms, starting from a covariant representation $\pi$ of $(\ca,\a)$ on a Hilbert space $\ch$, we construct a covariant representation $\pi_\infty$ of $(\ca_\infty,\a_\infty)$ on $\ch_\infty$. Then we construct a suitable representation $\r_\infty$ of $C_0(\bz_\infty)$ on $\ch_\infty$, which allows us to construct a representation $\s_\infty$ of $C_0(\bz_\infty;\ca_\infty)$ on $\ch_\infty$, and then a representation $\s_\infty \rtimes W_\infty$ of $C_0(\bz_\infty;\ca_\infty) \rtimes_\g \bz$, viz. a representation $\pi_\cc$ of $\cc$, on $\ch_\infty$, that we can restrict to $C^*(\ca,\a,\bn)$, and compress to a representation $\wh{\pi}$ on $\ch$ that satisfies property $(4)$ in Definition \ref{1.1}.

In order to the help the reader with the understanding of the following statements and proofs, we exhibit two tables with
the $C^*$-algebras considered, and their representations on the various Hilbert spaces
\begin{center}
  \begin{tabular}{ | c | c | c | c | c | }
    \hline
& $\Aut(\cdot)$ & $\ch_\infty$ & $\ch_u$ & $\ell^2(\bz;\ch_u)$ \\ \hline
$\ca_\infty$ & $\a_\infty$ & $\pi_\infty$ & $\pi_u$ & $\wt{\pi_u}$ \\ \hline
$C_0(\bz_\infty)$ & $\b$ & $\rho_\infty$ & $\rho_u$ & $\wt{\rho_u}$ \\ \hline
$C_0(\bz_\infty;\ca_\infty)$ & $\g\equiv \a_\infty\otimes\b$ & $\s_\infty$ & $\s_u$  & $\wt{\s_u}$ \\ \hline
$\ca_\infty \rtimes_{\a_\infty} \bz$ & - & $\pi_\infty \rtimes W_\infty$ & - & $\wt{\pi_u} \rtimes U$ \\ \hline
$C_0(\bz_\infty;\ca_\infty) \rtimes_{\g} \bz$ & - & $\s_\infty \rtimes W_\infty$ & - & $\chi\equiv \wt{\s_u} \rtimes U$ \\ \hline
$\cc \equiv \chi(C_0(\bz_\infty;\ca_\infty) \rtimes_{\g} \bz)$ & - & $\pi_\cc \equiv (\s_\infty \rtimes W_\infty)\circ \chi^{-1}$ & - & id \\ \hline
  \end{tabular}
\end{center}
and
\begin{center}
  \begin{tabular}{ | c | c | c | c | c | }
    \hline
& $\End(\cdot)$ & $\ch$ & $\ch_\infty$ & $p\ell^2(\bz;\ch_u)$ \\\hline
$\ca$ & $\a$ & $\pi$ & $\psi_0$  & $\zci_\ca \equiv \wt{\pi_u}\circ\f_{\infty0}(\cdot)p$ \\\hline
$C^*(\ca,\a,\bn)$ & - &  $\wh{\pi}\equiv S_{\infty0}^*\pi_\cc(\cdot)S_{\infty0}$ & $\pi_\cc|_{C^*(\ca,\a,\bn)}$ & id\\\hline
  \end{tabular}
\end{center}

Let  $(\ch,\piA,W)$ be a covariant representation of $(A,\a)$, and recall from Proposition \ref{prop:RepIndLim} that there exist $W_\infty\in\cu(\ch_\infty)$, and a covariant representation $(\ch_\infty, \ppinf, W_\infty)$ of $(\ca_\infty,\a_\infty)$, on  $\ch_\infty \equiv \varinjlim H_n$, the Hilbert space inductive limit of the inductive system \eqref{HilbertIndLim},
such that $\ppinf\circ \f_{\infty n}(a) S_{\infty n}  = S_{\infty n} \piA(a)$, for all $n\in\bn$, $a\in\ca$, and $W_\infty S_{\infty 0} = S_{\infty 0} W$. 

We now construct a representation $\rrinf$ of $C_0(\bz_\infty)$ on $\ch_\infty$ such that $[\ppinf(a),\rrinf(f)]=0$, for all $a\in\ca_\infty$, $f\in C_0(\bz_\infty)$. 


\begin{proposition} \label{prop:spectralFamily}
Set $P_0:= S_{\infty 0}S_{\infty 0}^*$, $P_n:= \Ad(W_\infty^n)(P_0)$, $n\in\bz$. Then
\begin{itemize}
	\item[$(1)$] $\{ P_n:n\in\bz\}$ is a decreasing family of projections in $\cb(\ch_\infty)$,
	
	\item[$(2)$] there exists $P_{+\infty}:= \lim_{n\to+\infty} P_n$, in the strong operator topology of $\cb(\ch_\infty)$,
	
	\item[$(3)$] $\lim_{n\to-\infty} P_n = 1$, in the strong operator topology of $\cb(\ch_\infty)$,
	
	\item[$(4)$] $\{ P_n : n\in\bz_\infty \} \subset \ppinf(\ca_\infty)'$.
	\end{itemize}
\end{proposition}
\begin{proof}
$(1)$ Let $n\in\bz$. If $n\geq0$, then 
$$
P_n = W_\infty^nS_{\infty 0}S_{\infty 0}^*W_\infty^{n*} = S_{\infty 0}W^nW^{n*}S_{\infty 0}^* \geq S_{\infty 0}W^{n+1}(W^{*})^{n+1}S_{\infty 0}^* = P_{n+1}.
$$ 
If $n=-k\leq 0$, then 
$$
P_n = W_\infty^{*k}S_{\infty 0}S_{\infty 0}^*W_\infty^{k} = S_{\infty k}S_{\infty k}^* = S_{\infty,k+1}WW^*S_{\infty,k+1}^* \leq S_{\infty,k+1}S_{\infty,k+1}^* = P_{n-1}.
$$

\noindent $(2)$ follows from $(1)$.

\noindent $(3)$ We have to prove that $\lim_{k\to+\infty} S_{\infty k}S_{\infty k}^* = 1$, in the strong operator topology, and it suffices to prove it on the dense subset of $\ch_\infty$ spanned by $\{ S_{\infty n}\xi : n\in\bn, \xi\in\ch\}$. Let us fix $n\in\bn$, $\xi\in\ch$, and compute, for $k>n$, $S_{\infty k}S_{\infty k}^*S_{\infty n}\xi = S_{\infty k}S_{\infty k}^*S_{\infty k}S_{kn}\xi = S_{\infty k}S_{kn}\xi = S_{\infty n}\xi$, and the thesis follows.

\noindent $(4)$  Let us first prove that $\ppinf(x) P_0 = P_0 \ppinf(x)$ for $x \in \ca_\infty$. It suffices to show the equality for $x\in \{ \f_{\infty n}(a) : n\in\bn,a\in\ca\}$. We have, from equation \eqref{allaccia},
\begin{align*}
	\ppinf \circ \f_{\infty n}(a)P_0 & = \ppinf\circ \f_{\infty n}(a)S_{\infty 0}S_{\infty 0}^* = \ppinf\circ \f_{\infty n}(a)S_{\infty n}W^nS_{\infty 0}^* \\
	& = S_{\infty n} \piA(a) W^n W^{*n} S_{\infty n}^*  = S_{\infty n} W^n W^{*n} \piA(a) S_{\infty n}^* \\
	& = S_{\infty 0}  W^{*n}  S_{\infty n}^* \ppinf \circ \f_{\infty n}(a)  =  P_0 \ppinf \circ \f_{\infty n}(a).
\end{align*}
Then, for any $x \in \ca_\infty$, $k\in\bz$, 
\begin{align*}
	\ppinf(x) P_k & = \ppinf(x)W_\infty^kP_0W_\infty^{*k} = W_\infty^k\ppinf(\a_\infty^{-k}(x)) P_0W_\infty^{*k} \\
	& = W_\infty^kP_0\ppinf(\a_\infty^{-k}(x)) W_\infty^{*k} = W_\infty^kP_0W_\infty^{*k} \ppinf(x) = P_k \ppinf(x).
\end{align*}
Finally, $P_{+\infty} \in \ppinf(\ca_\infty)'$, because of $(2)$.
\end{proof}


\begin{proposition} \label{prop:repC0}
There exists a representation $\rrinf$ of $C_0(\bz_\infty)$ on $\ch_\infty$, such that, for any $f\in C_0(\bz_\infty)$, 
\begin{align*}
	\rrinf(f) & \in\ppinf(\ca_\infty)', \\
	\rrinf(\betaC(f)) & = W_\infty \rrinf(f)W_\infty^*.
\end{align*}	
\end{proposition}
\begin{proof}
Set $E_n:= P_n-P_{n+1}$, $n\in\bz$, $E_{+\infty}:=P_{+\infty}$. Then, $\{ E_n: n\in\bz_\infty\}$ is a spectral family on $\ch_\infty$, and $E_{n+1} = W_\infty E_n W_\infty^*$, $n\in\bn$, $E_{+\infty} = W_\infty E_{+\infty} W_\infty^*$. Define, for $f\in C_0(\bz_\infty)$, $\rrinf(f):= \sum_{n\in\bz_\infty} f(n)E_n$, where the series converges in the strong operator topology of $\cb(\ch_\infty)$. Then, $\rrinf$ is a representation of $C_0(\bz_\infty)$ on $\ch_\infty$, such that $\rrinf(f)\in\ppinf(\ca_\infty)'$, for any $f\in C_0(\bz_\infty)$, and $\rrinf(\betaC(f)) = W_\infty \rrinf(f)W_\infty^*$, $f\in C_0(\bz_\infty)$. 
\end{proof}


\begin{proposition} \label{prop:Sigma}
\begin{itemize}
\item[$(1)$] There is a unique representation $\ssinf$ of $C_0(\bz_\infty;\ca_\infty)$ on $\ch_\infty$, such that $\ssinf(a\otimes f) = \ppinf(a)\rrinf(f)$, $a\in\ca_\infty$, $f\in C_0(\bz_\infty)$. Moreover, $\ssinf(\gammaCA(g)) = W_\infty \ssinf(g)W_\infty^*$, $g\in C_0(\bz_\infty;\ca_\infty)$.
	
\item[$(2)$] There is a unique representation $\ssW$ of $C_0(\bz_\infty;\ca_\infty)\rtimes_{\gammaCA} \bz$ on $\ch_\infty$ such that $\ssW(g\d_n) = \ssinf(g)W_\infty^n$, $g\in C_0(\bz_\infty;\ca_\infty)$, $n\in\bz$.
\end{itemize}
\end{proposition}
\begin{proof}
$(1)$ This follows from \cite{Take},  Proposition IV.4.7.
	
\noindent $(2)$ This follows from \cite{Ped}, Proposition 7.6.4.
\end{proof}


Let us set $\pi_\cc := \ssW \circ \chi^{-1}$, which is a representation of $\cc$ on $\ch_\infty$.

\begin{proposition} \label{prop:repr}
\begin{itemize}

	\item[$(1)$]  $\pi_\cc( x ) = \ppW( x )$, for all $x\in\ca_\infty \rtimes_{\a_\infty} \bz \equiv \langle\, \ppcov(\ca_\infty), \uu \,\rangle$, 
	
	\item[$(2)$] $\pi_\cc(p) = P_0 = S_{\infty 0}S_{\infty 0}^*$.
	\end{itemize}

\end{proposition}
\begin{proof}
It follows from Proposition \ref{prop:reprOfCrossed} that, for $\sum_{n\in\bz} (a_n\otimes f_n) \dd_n \in C_c( C_0(\bz_\infty;\ca_\infty), \bz, \gammaCA)$, we have $\chi( \sum_{n\in\bz} (a_n\otimes f_n) \dd_n ) = \sum_{n\in\bz} \ppcov(a_n)\rrcov(f_n)u^n$, so that 
\begin{align*}
	\pi_\cc \bigg( \sum_{n\in\bz} \ppcov(a_n)\rrcov(f_n) \uu^n \bigg) & = \ssW \bigg( \sum_{n\in\bz} (a_n\otimes f_n) \dd_n \bigg) \\
	& = \sum_{k\in\bz} \ssinf( a_n \otimes f_n) W_\infty^n = \sum_{k\in\bz} \ppinf( a_n) \rrinf( f_n) W_\infty^n.
\end{align*}

\noindent $(1)$ Indeed, with $\set{e_n:n\in\bn}$ an approximate unit of $C_0(\bz_\infty)$, we get, for all $a\in\ca_\infty$, $k\in\bz$, 
\begin{align*}
	\pi_\cc( \ppcov(a) \uu^k ) & = \lim_{n\to\infty}  \pi_\cc( \ppcov(a)\rrcov(e_n) \uu^k)      = \lim_{n\to\infty}  \ppinf(a) \rrinf(e_n)W_\infty^k \\
	& = \ppinf(a)W_\infty^k = \ppW( \ppcov(a) \uu^k ),
\end{align*}
and the thesis follows.

\noindent $(2)$ If $f(n) = 
\begin{cases}
	0, & n<0,\\
	1, & n\geq 0,
\end{cases}\ $ then $\widehat{\chi}(p)  = \widehat{\chi}(\rrcov(f))  =  \rrinf(f) = P_0$.
\end{proof}


Let us still denote by $\pi_\cc$ the restriction of $\pi_\cc$ to the subalgebra $C^*(\ca,\a,\bn) \equiv \langle\, p\ppcov(\ca_\infty)p, p\uu p \,\rangle$ of $\cc \equiv \langle\, \ppcov(\ca_\infty), \rrcov(C_0(\bz_\infty)), \uu \,\rangle$.

\begin{theorem} \label{prop:crossedProd}
	$C^*(\ca,\a,\bn)$ satisfies all the properties in Definition \ref{1.1}, namely  is  the crossed product of $\ca$ with $\bn$ by $\a$. 
\end{theorem}
\begin{proof}
As it was already noticed, property $(1)$ in Definition \ref{1.1} follows by definition, while properties $(2)$ and $(3)$  have been proved in Proposition \ref{prop:generate}. 

It remains to prove property $(4)$. Let $(\ch,\piA,W)$ be a covariant representation of $(A,\a)$, and recall from Proposition \ref{prop:RepIndLim} that there exist $W_\infty\in\cu(\ch_\infty)$, and a covariant representation $(\ch_\infty, \ppinf, W_\infty)$ of $(\ca_\infty,\a_\infty)$, on  $\ch_\infty \equiv \varinjlim H_n$, the Hilbert space inductive limit of the inductive system \eqref{HilbertIndLim},
such that $\ppinf\circ \f_{\infty n}(a) S_{\infty n}  = S_{\infty n} \piA(a)$, for all $n\in\bn$, $a\in\ca$, and $W_\infty S_{\infty 0} = S_{\infty 0} W$. 

Let $\pi_\cc$ be the representation  of $C^*(\ca,\a,\bn)$ on $\ch_\infty$ constructed in Proposition \ref{prop:repr}.
Let us now prove that $P_0\in \pi_\cc(C^*(\ca,\a,\bn))'$, that is $\pi_\cc(C^*(\ca,\a,\bn)) S_{\infty 0}\ch \subset S_{\infty 0}\ch$. Because of Proposition \ref{prop:generate} it is enough to prove that $\pi_\cc(t)S_{\infty 0}\ch \subset S_{\infty 0}\ch$, $\pi_\cc(t^*)S_{\infty 0}\ch \subset S_{\infty 0}\ch$, and $\pi_\cc(\zci_\ca(a))S_{\infty 0}\ch \subset S_{\infty 0}\ch$, for all $a\in\ca$. Indeed, for all $a\in\ca$, $\xi\in\ch$, we have
\begin{align*}
	\pi_\cc(t)S_{\infty 0}\xi & = \pi_\cc(p\uu p)S_{\infty 0}\xi = P_0W_\infty P_0S_{\infty 0}\xi \in S_{\infty 0}\ch, \\
	\pi_\cc(t^*)S_{\infty 0}\xi & = \pi_\cc(p\uu^* p)S_{\infty 0}\xi = P_0W_\infty^* P_0S_{\infty 0}\xi \in S_{\infty 0}\ch, \\
	\pi_\cc( \zci_\ca(a) )S_{\infty 0}\xi & = \pi_\cc\circ \ppcov \circ \f_{\infty 0} (a) P_0S_{\infty 0}\xi =  \ppinf \circ \f_{\infty 0} (a) S_{\infty 0}\xi \\
	& = S_{\infty 0} \piA(a)\xi \in S_{\infty 0}\ch.
\end{align*}
Recall from the proof of Proposition \ref{prop:RepIndLim} that there is a representation $\piAinf_0$ of $\ca$ on $\ch_\infty$ such that $\piAinf_0(a)S_{\infty 0} = S_{\infty 0}\piA(a)$, $a\in\ca$, and $\ppinf\circ\f_{\infty 0} = \piAinf_0$. Finally, define 
\begin{align*}
	\widehat{\piA}(x):= S_{\infty 0}^* \pi_\cc(x)S_{\infty 0},\quad x\in C^*(\ca,\a,\bn),
\end{align*}
which is a representation of $C^*(\ca,\a,\bn)$ on $\ch$, because $P_0\in \pi_\cc(C^*(\ca,\a,\bn))'$. Then, 
\begin{align*}
	\widehat{\piA}(t) & = S_{\infty 0}^* \pi_\cc(t)S_{\infty 0} = S_{\infty 0}^* P_0W_\infty P_0S_{\infty 0} = S_{\infty 0}^* W_\infty S_{\infty 0} \\
	& =  S_{\infty 0}^* S_{\infty 0}W = W, 
\end{align*}
and, for all $a\in\ca$, 
\begin{align*}
	\widehat{\piA}(\zci_\ca(a)) & = S_{\infty 0}^* \pi_\cc( \ppcov \circ \f_{\infty 0} (a) p) S_{\infty 0} = S_{\infty 0}^* \ppinf \circ   \f_{\infty 0} (a) S_{\infty 0}S_{\infty 0}^* S_{\infty 0} \\
	& = S_{\infty 0}^* \piAinf_0 (a) S_{\infty 0} = S_{\infty 0}^* S_{\infty 0}\piA(a) = \piA(a).
\end{align*}
\end{proof}

\subsection{An example: the noncommutative torus}
As mentioned in the introduction, the crossed product $\ca \rtimes_\a \bn$ given in Definition \ref{1.1}  coincides with   a reduction by a projection of the ordinary crossed product when $\alpha$ is an automorphism. 
We now give two equivalent descriptions of $\ca \rtimes_\a \bn$, when $\ca=C(\br/\bz)$ and $\alpha$ is a rotation by $2\pi \theta$, where $\theta$ is irrational.

The first description is the following.
As it is known, the noncommutative torus $A_\theta$ can be described as the crossed product $C(\br/\bz)\rtimes_{\alpha_\theta}\bz$, where $(\alpha_\theta(f))(t)=f(t-\theta)$. Given the Hilbert space $H=\ell^2(\bz,L^2(\br/\bz))$,  the representation $\pi:C(\br/\bz)\to \cb(H)$, $(\pi(f)\xi)(n)=\alpha_\theta^{-n}(f)\xi(n)$ and the unitary $V$ acting on $H$ as $(V\xi)(n)=\xi(n-1)$,  $A_\theta$ can be identified with the $C^*$-algebra generated by $V$ and $\pi(C(\br/\bz))$ on the Hilbert space $H$. Since $C(\br/\bz)$ is generated as a $C^*$-algebra by the unitary $U_0=\exp(2\pi i t)$, $A_\theta$ is generated by the unitary $V$ and the unitary $U$ given by $(U\xi)(n)=\exp(2\pi i n\theta)U_0\xi(n)$.
It is easy to check that $UV=\exp(2\pi i\theta)VU$.

Since $\alpha_\theta$ is an automorphism,  Theorem \ref{prop:crossedProd} implies that $C(\br/\bz)\rtimes_{\alpha_\theta}\bn$ is the reduction of $C(\br/\bz)\rtimes_{\alpha_\theta}\bz$ by the projection $p$ on the Hilbert space $H_+=\ell^2(\bn_0,L^2(\br/\bz))$.
We have proved the following theorem.
\begin{theorem}
The C$^*$-algebra $C(\br/\bz)\rtimes_{\alpha_\theta}\bn$ can be identified with the $C^*$-algebra generated by the unitary $U$ and the isometry $pVp$ acting on $H_+$.
\end{theorem}


We now provide a description of  $\ca \rtimes_\a \bn$ as a universal object.
\begin{theorem}
The $C^*$-algebra $C(\br/\bz)\rtimes_{\alpha_\theta}\bn$ coincides with the universal $C^*$-algebra generated by a unitary $U$ and an isometry $V$ satisfying the conditions $UV=\exp(2\pi i\theta)VU$. 
\end{theorem}

\begin{proof}
By definition, the universal $C^*$-algebra generated by a unitary $U$ and an isometry $V$ satisfying the conditions $UV=\exp(2\pi i\theta)VU$ 
 is the unique $C^*$-algebra $B$ satisfyng the following universal property: for any triple $(\ch,u,v)$, where $\ch$ is a Hilbert space, $u$ is a unitary and $v$ is an isometry acting on $\ch$ satisfying $uv=\exp(2\pi i\theta)vu$, there exists a representation $\pi:B\to \cb(\ch)$ such that  $\pi(U)=u$ and $\pi(V)=v$.

By definition, also $C(\br/\bz)\rtimes_{\alpha_\theta}\bn$ satisfies a universal property, given by properties $(1) - (4)$ of Definition 2.4. Therefore, given a triple  $(\ch,u,v)$ as above,  we get indeed a covariant representation $(\ch,\rho, v)$ of $(C(\br/\bz),\alpha_\theta)$, where we set $\rho(f)=f(u)$, in fact the commutation relations imply that $v^kv^{*k}u=u v^kv^{*k}$. The properties of $C(\br/\bz)\rtimes_{\alpha_\theta}\bn$ imply the thesis.
\end{proof}

\begin{remark}
If $\theta$ is rational, the projection $p$ in the first description is the identity and, therefore,  $C(\br/\bz)\rtimes_{\alpha_\theta}\bn$ coincides with 
$A_\theta$.
\end{remark}


\section{Some results on semifinite spectral triples}

In this section we discuss some generalizations of results well-known for type I spectral triples. Some of these results have already been proved in \cite{Jordans} and some are new.

First of all we recall the following definitions:


\begin{definition}
Let $(\cam,\t)$ be a von Neumann algebra with a  normal semifinite faithful (n.s.f.) trace,  $T\, \widehat{\in}\, \cam$   a self-adjoint operator\footnote{By $T\, \widehat{\in}\, \cam$ we mean that the operator $T$ is affiliated with $\cam$. Another common notation is $T\, \eta\, \cam$.}. We use the notation $e_T(\O)$ for the spectral projection of  $T$ relative to the measurable set $\O\subset \br$, $\l_t(T):=\t(e_{|T|}(t,+\infty))$, $\Lambda_t(T):=\t(e_{|T|}[0,t))$, $\mu_t(T):=\inf \{s>0: \lambda_T(s)\leq t\}$,   $t>0$. The operator $T$ is said to be $\t$-measurable if $\l_t(T) \to 0$, $t\to+\infty$, and $\t$-compact if $\m_t(T) \to 0$, $t\to+\infty$, or equivalently, $\l_t(T)<+\infty$, $\forall\; t>0$.
\end{definition}


\begin{definition} 
Let $\ca$ be a unital $C^*$-algebra. An odd semifinite spectral triple $(\cl,\ch,D; \cam,\t)$  on $\ca$, with respect to a semifinite von Neumann algebra $\cam\subset\cb(\ch)$ endowed with a n.s.f. trace $\t$,  is given by a unital, norm-dense, $^*$-subalgebra $\cl\subset\ca$, a  (separable) Hilbert space $\ch$, a faithful representation $\pi:\ca\to\cb(\ch)$ such that $\pi(\ca)\subset\cam$, and an unbounded self-adjoint operator $D\, \widehat{\in}\, \cam$ such that
\begin{itemize}
\item[$(1)$] $(1+D^2)^{-1}$ is a $\t$-compact operator, \textit{i.e.} $\l_t((1+D^2)^{-1})<+\infty$, $\forall\; t>0$ or, equivalently, $\Lambda_t(D)<+\infty$, $\forall \; t>0$,
\item[$(2)$] $\pi(a)(\dom D) \subset \dom D$, and $[D,\pi(a)] \in\cam$, for all $a\in\cl$.
\end{itemize}

\par\noindent

The spectral triple $(\cl,\ch,D; \cam,\t)$ is even if, in addition,
\begin{itemize}
\item[$(3)$] there is a  self-adjoint unitary operator (\textit{i.e.} a $\bz_2$-grading) $\G\in\cam$ such that $\pi(a)\G = \G\pi(a)$, $\forall a\in\ca$, and $D\G=-\G D$. 
\par
The spectral triple $(\cl,\ch,D; \cam,\t)$ is finitely summable if, in addition,
\item[$(4)$] there exists a $\d>0$ such that $\tau((1+D^2)^{-\d/2})<+\infty$.
\end{itemize}
\end{definition}

\begin{definition}
Given a finitely summable semifinite spectral triple $(\cl,\ch,D; \cam,\t)$, the number $d=\inf\{\a>0:\tau((1+D^2)^{-\a/2})<+\infty\}$ is called the metric or Hausdorff dimension of the triple, since it is the unique exponent, if any, such that the logarithmic Dixmier trace is finite non-zero on $(1+D^2)^{-\a/2}$ (cf. \cite{GuIs09}, Theorem 2.7).
\end{definition}

We note that the usual definition of spectral triple, which was recalled in Definition \ref{deftripla}, can be recovered by taking $\cam=\cb(\ch)$.


\begin{proposition}  \label{dimension-Lambda}
Let $(\cl,\ch,D; \cam,\t)$ be a finitely summable semifinite spectral triple. Then $d= \limsup_{t\to\infty}\frac{\log\Lambda_t(D)}{\log t}$.
\end{proposition}  
\begin{proof}
We first observe that, by \cite{FaKo}, Proposition 2.7,
$$
 \tau((1+D^2)^{-\a/2})=\int_0^{+\infty}\mu_t((1+D^2)^{-\a/2})dt
 = \int_0^{+\infty}\mu^\a_t((1+D^2)^{-1/2})dt.
$$ 
Therefore,
\begin{align*}
 d&=\left(\liminf_{t\to\infty}\frac{\log\m_t((1+D^2)^{-1/2})}{\log(1/t)}\right)^{-1}
 =\limsup_{s\to0}\frac{\log\lambda_s((1+D^2)^{-1/2})}{\log(1/s)}\\
 &=\limsup_{t\to\infty}\frac{\log\Lambda_t((1+D^2)^{1/2})}{\log t}
 =\limsup_{t\to\infty}\frac{\log\Lambda_t(D)}{\log t},
\end{align*}
where the first equality follows by  \cite{GuIs09} Theorem 1.4, the second by \cite {GuIs05} Proposition 1.13, the third by definition of $\Lambda$, the last by simple estimates.
\end{proof}


\subsection{The case of the tensor product}
 
Let us recall the definition of tensor product of semifinite spectral triples.  
 
\begin{definition}
Let $\ca_1,\ca_2$ be unital $C^*$-algebras, with respective semifinite spectral triples $\ct_1:=(\cl_1,\ch_1,D_1,\G_1;\cam_1,\t_1)$, $\ct_2:=(\cl_2,\ch_2,D_2,\G_2;\cam_2,\t_2)$, and define $\ct_1\times \ct_2 \equiv (\cl,\ch,D,\G;\cam,\t)$ as follows:
\\
 if $\ct_1$  and $\ct_2$ are both even
$$\begin{array}{ccc}
\cl:=\cl_1 \odot \cl_2, &
\ch:= \ch_1 \otimes \ch_2, &
D :=  D_1 \otimes I_2  + \G_1 \otimes D_2, \\
\G:=\G_1\otimes \G_2, &
\cam := \cam_1 \otimes \cam_2, &
\t := \t_1 \otimes \t_2,
\end{array}$$
\\
 if $\ct_1$  is  even, and $\ct_2$ is odd,
$$\begin{array}{ccc}
\cl:=\cl_1 \odot \cl_2,  &
\ch:= \ch_1 \otimes \ch_2, &
D :=  D_1 \otimes I_2 + \G_1 \otimes D_2, \\
\G:=I_1\otimes I_2, &
\cam := \cam_1 \otimes \cam_2, &
\t := \t_1 \otimes \t_2,
\end{array}$$
\\
 if $\ct_1$  is  odd, and $\ct_2$ is even,
$$\begin{array}{ccc}
\cl:=\cl_1 \odot \cl_2,  &
\ch:= \ch_1 \otimes \ch_2, &
D :=  D_1 \otimes \G_2 + I_1 \otimes D_2, \\
\G:=I_1\otimes I_2, &
\cam := \cam_1 \otimes \cam_2, &
\t := \t_1 \otimes \t_2,
\end{array}$$
\\
 if $\ct_1$  and $\ct_2$ are both odd,
$$\begin{array}{ccc}
\cl:=\cl_1 \odot \cl_2, &
\ch:= \ch_1 \otimes \ch_2 \otimes\bc^2,  &
D :=  D_1 \otimes I_2 \otimes \eps_1 + I_1 \otimes D_2 \otimes \eps_2, \\
\G:=I_1\otimes I_2\otimes \eps_3, &
\cam := \cam_1 \otimes \cam_2\otimes M_2(\bc), &
\t := \t_1 \otimes \t_2\otimes Tr,
\end{array}$$

where $\eps_1$, $\eps_2$, $\eps_3$
are the Pauli matrices, see \eqref{pauli}. 
\end{definition}


\begin{proposition}  \label{tensorProductTriple}
Let $\ca_1,\ca_2$ be unital $C^*$-algebras, with respective semifinite spectral triples $\ct_1:=(\cl_1,\ch_1,D_1,\G_1;\cam_1,\t_1)$, $\ct_2:=(\cl_2,\ch_2,D_2,\G_2;\cam_2,\t_2)$. Then $\ct_1 \times \ct_2$ is a semifinite spectral triple on the spatial tensor product $\ca_1 \otimes \ca_2$. Moreover, the Hausdorff dimension $d$ of $\ct_1 \times \ct_2$ satisfies $d\leq d_1 + d_2$, where $d_1, d_2$ are the Hausdorff dimensions of the factor spectral triples. Finally, if  $\displaystyle\lim_{t\to\infty}\frac{\log\Lambda_t(D_1)}{\log t}$ exists, the equality $d= d_1+d_2$ holds.
\end{proposition}  
\begin{proof} 
In case $\ct_1$ and $\ct_2$ are not both odd, the result is proved in \cite{Jordans}, Theorem 2.13, and Lemma 2.19. In the remaining case, one can proceed analogously.
We now give an alternative proof of the formula for the Hausdorff dimension, valid in all cases. 
Since $D^2=D_1^2\otimes I+I\otimes D_2^2$, in all cases, if $d$ denotes the dimension of $(\cl,\ch,D; \cam,\t)$, we have that 
\begin{align*}
d & = \limsup_{t\to\infty}\frac{\log\Lambda_t(D)}{\log t}
= \limsup_{t\to\infty}\frac{\log\t(e_{D}(-t,t))}{\log t} \\
& = \limsup_{t\to\infty}\frac{\log\t(\chi_{[0,t^2)}(D_1^2\otimes I+I\otimes D_2^2))}{\log t}.
\end{align*}
If $\s_i$ denotes the spectrum of $D_i$, $i=1,2$, the representations of $C_0(\s_i)$ on $\ch_i$ with image in $\cam_i$ given by functional calculus, $i=1,2$, together with the Radon measures $\nu_i$ on $\s_i$ induced by the traces $\tau_i$, $i=1,2$, give rise to a representation $j$ of $C_0(\s_1\times\s_2)$ on $\ch_1\otimes\ch_2$ with image in $\cam_1\otimes\cam_2$ together with the Radon measure $\nu:=\nu_1\otimes\nu_2$ on $\s_1\times\s_2$ induced by the trace $\tau:=\tau_1\otimes\tau_2$ such that $j(f_1\otimes f_2)=f_1(D_1)\otimes f_2(D_2)$ and $\int f_1\otimes f_2 d\nu=\tau_1(f_1(D_1)) \tau_2(f_2(D_2))$. Then, denoting by $B_r$ the disk of radius $r$ centered in the origin of the plane, and by $Q_r$ the square $[-r,r]\times[-r,r]$ in the plane,
$$
\chi_{[0,t^2)}(D_1^2\otimes I+I\otimes D_2^2) = j(\chi_{B_t}).
$$
Then the inclusions $Q_{t/\sqrt2}\subset B_t\subset Q_t$ give the inequalities
$$
\tau_1(\Lambda_{t/\sqrt2}(D_1))\cdot\tau_2(\Lambda_{t/\sqrt2}(D_2)) \leq \nu(Q_{t/\sqrt2}) \leq 
\nu(B_t)\leq\nu(Q_t)\leq\tau_1(\Lambda_t(D_1))\cdot\tau_2(\Lambda_t(D_2)),
$$
from which we get
\begin{align*}
\liminf_{t\to\infty}\frac{\log\Lambda_t(D_1)}{\log t}+\limsup_{t\to\infty}\frac{\log\Lambda_t(D_2)}{\log t}
& \leq \limsup_{t\to\infty}\frac{\log\Lambda_t(D)}{\log t} \\
& \leq \limsup_{t\to\infty}\frac{\log\Lambda_t(D_1)}{\log t}+\limsup_{t\to\infty}\frac{\log\Lambda_t(D_2)}{\log t}.
\end{align*}
\end{proof}


\subsection{The cases of the crossed products}\label{subsec:CrossedProd}

Let $\ca$ be a unital $C^*$-algebra, $\a\in \Aut(\ca)$ a unital automorphism, and $(\cl,\ch,D;\cam,\t)$ a semifinite spectral triple on $\ca$. Assume that $\a$ is Lip-bounded, that is $\a(\cl)\subset\cl$, and, for any $a\in\cl$, $\sup_{n\in\bz} \| [D,\a^{-n}(a)] \| < \infty$. Then, following \cite{BMR}, we can construct a semifinite spectral triple $(\cl_\rtimes,\ch_\rtimes,D_\rtimes;\cam_\rtimes,\t_\rtimes)$ on the crossed product $C^*$-algebra $\ca\rtimes_\a \bz = \langle \wt{\pi_u}(\ca),U \rangle$, which is defined as follows:

\begin{itemize}
\item[$(1)$] if $(\cl,\ch,D,\G;\cam,\t)$ is even, 
\begin{flalign*}
& \cl_\rtimes := {}^*\mathrm{alg}(\wt{\pi_u}(\cl),U), & \qquad  & \ch_\rtimes   := \ch \otimes \ell^2(\bz), & \\
& D_\rtimes := D \otimes I + \G \otimes D_\bz, & \qquad  & \G_\rtimes  := I\otimes I, & \\
& \cam_\rtimes := \cam \otimes \cb(\ell^2(\bz)), & \qquad & \t_\rtimes  := \t \otimes Tr, &
\end{flalign*}
where ${}^*\mathrm{alg}(\wt{\pi_u}(\cl),U)$ is the $^*$-algebra generated by $\wt{\pi_u}(\cl)$ and $U$, $(D_\bz \xi)(n) := n\xi(n)$, $\forall \xi\in\ell^2(\bz)$, and $Tr$ is the usual trace on $\cb(\ell^2(\bz))$,

\item[$(2)$] if $(\cl,\ch,\G;\cam,\t)$ is odd,
\begin{flalign*}
& \cl_\rtimes := {}^*\mathrm{alg}(\wt{\pi_u}(\cl),U), & \qquad  & \ch_\rtimes   := \ch \otimes \ell^2(\bz) \otimes \bc^2, & \\
& D_\rtimes := D \otimes I \otimes \eps_1 + I \otimes D_\bz \otimes \eps_2, & \qquad  & \G_\rtimes  := I \otimes I \otimes \eps_3, & \\
& \cam_\rtimes := \cam \otimes \cb(\ell^2(\bz)) \otimes M_2(\bc), & \qquad & \t_\rtimes  := \t \otimes Tr \otimes tr, &
\end{flalign*}
where $tr$ is the normalized trace on $M_2(\bc)$.
\end{itemize}


In case $\a$ satisfies a weaker condition, we have the following result.

\begin{definition}\label{semiequi}
Let $\ca$ be a unital C*-algebra, $\a\in \Aut(\ca)$ a unital automorphism, $(\cl,\ch,D)$ a spectral triple on $\ca$ such that $\a(\cl)\subset\cl$.
The automorphism is said to be Lip-semibounded if
$$
\sup_{n\in\bn} \| [D,\a^{-n}(a)] \| < \infty, \qquad \forall a\in\cl.
$$
\end{definition}


\begin{proposition}  
Let $\ca$ be a unital $C^*$-algebra, $\a\in \Aut(\ca)$ a unital automorphism, $(\cl,\ch,D;\cam,\t)$ a semifinite spectral triple on $\ca$, and assume $\a$ is Lip-semibounded. Then we can construct a semifinite spectral triple $(\cl_\rtimes,\ch_\rtimes,D_\rtimes;\cam_\rtimes,\t_\rtimes)$ on the crossed product $C^*$-algebra $\ca\rtimes_\a \bn = \langle \zci_\ca(\ca),t \rangle$, which is defined as follows:

\begin{itemize}
\item[$(1)$] if $(\cl,\ch,D,\G;\cam,\t)$ is even, 
\begin{flalign*}
& \cl_\rtimes := {}^*\mathrm{alg}(\zci_\ca(\cl),t), & \qquad  & \ch_\rtimes   := \ch \otimes \ell^2(\bn), & \\
& D_\rtimes := D \otimes I + \G \otimes D_\bn, & \qquad  & \G_\rtimes  := I\otimes I, & \\
& \cam_\rtimes := \cam \otimes \cb(\ell^2(\bn)), & \qquad & \t_\rtimes  := \t \otimes Tr, &
\end{flalign*}
where ${}^*\mathrm{alg}(\zci_\ca(\cl),t)$ is the $^*$-algebra generated by $\zci_\ca(\cl)$ and $t$, $(D_\bn \xi)(n) := n\xi(n)$, $\forall \xi\in\ell^2(\bn)$, and $Tr$ is the usual trace on $\cb(\ell^2(\bn))$,

\item[$(2)$] if $(\cl,\ch,\G;\cam,\t)$ is odd,
\begin{flalign*}
& \cl_\rtimes := {}^*\mathrm{alg}(\zci_\ca(\cl),t), & \qquad  & \ch_\rtimes   := \ch \otimes \ell^2(\bn) \otimes \bc^2, & \\
& D_\rtimes := D \otimes I \otimes \eps_1 + I \otimes D_\bn \otimes \eps_2, & \qquad  & \G_\rtimes  := I \otimes I \otimes \eps_3, & \\
& \cam_\rtimes := \cam \otimes \cb(\ell^2(\bn)) \otimes M_2(\bc), & \qquad & \t_\rtimes  := \t \otimes Tr \otimes tr, &
\end{flalign*}
where $tr$ is the normalized trace on $M_2(\bc)$.

\noindent Moreover, in both cases, if $\som$ is the dimension of the original spectral triple, then the dimension of the new spectral triple is $\som+1$. 
\end{itemize}
\end{proposition}  
\begin{proof} 
We only prove the even case, the odd case being similar. Let us first observe that, since $\a$ is an automorphism, $\ca_\infty = \ca$, $\a_\infty=\a$, and $\zci_\ca(a) = \ppcov(a)p$, $\forall a\in\ca$. Let $\pi:\ca\to\cb(\ch)$ be the representation implied by the spectral triple $(\cl,\ch,D,\G;\cam,\t)$, and consider $(\wt{\pi}(a)\xi)(n) := \pi(\a^{-n}(a))\xi(n)$, $\forall a\in\ca$, $\xi\in\ch\otimes\ell^2(\bn)$, $n\in\bn$, which is a representation of $\ca$ on $\ch\otimes\ell^2(\bn)$, and the shift operator 
$$
(W\xi)(n) :=
\begin{cases}
0, & n=0,\\
\xi(n-1), & n\geq 1.
\end{cases}\ 
$$ 
Then, it is easy to see that $(\ch\otimes\ell^2(\bn),\wt{\pi},W)$ is a covariant representation of $(\ca,\a,\bn)$ on $\ch\otimes\ell^2(\bn)$, in the sense of Definition \ref{Def:CovRep}. Therefore it induces a non-degenerate representation $\widehat{\pi}$ of $\ca\rtimes_\a \bn = \langle \zci(\ca), t \rangle$ on $\ch\otimes\ell^2(\bn)$, such that $\widehat{\pi}\circ \zci_\ca = \wt{\pi}$, and $\wh{\pi}(t)=W$. Hence $\widehat{\pi}(\ca\rtimes_\a \bn) \subset \cam_\rtimes$, while the facts that $D_\rtimes\widehat{\in} \cam_\rtimes$, and $(1+D_\rtimes^2)^{-1}$ is $\t_\rtimes$-compact follow from Proposition \ref{tensorProductTriple}. It remains to prove that $\| [ D_\rtimes, \widehat{\pi}(a) ] \| < \infty$, $\forall a\in\cl_\rtimes$. Since the commutators $[\G\otimes D_\bn,\widehat{\pi}(a)]$ and $[D\otimes I,W]$ vanish, while $\| [\G\otimes D_\bn, W] \| \leq 1$, it is enough to estimate the commutators $\| [D\otimes I,\widehat{\pi}(a) ] \| = \| \diag \{ [D,\pi(\a^{-n}(a)) ] : n\in\bn \} \| = \sup_{n\in\bn} \| [D,\pi(\a^{-n}(a)) ] \| < \infty$, and the claim follows.

\noindent We now prove the statement about the dimension, which in turn implies (again) the $\tau$-compactness of the resolvent.
By Proposition \ref{dimension-Lambda}, the Hausdorff dimension of $D_\rtimes$ is given by
\begin{displaymath}
	\limsup_{t\to +\infty}\frac{\log(\Lambda_t(D_\rtimes))}{\log t}.
\end{displaymath}
We observe that $\Lambda_{t}(D_\bn)=[t]$ and thus 
$$
\limsup_{t\to \infty} \frac{\log \Lambda_t (D_\bn)}{\log t}=\lim_{t\to \infty} \frac{\log ([t])}{\log t }=1.
$$
Now by applying Proposition \ref{tensorProductTriple} we are done.
\end{proof}


The next result has to do with the case of crossed products with respect to endomorphisms.

\begin{theorem} \label{triple-cross-prod-N}
Let $\ca$ be a unital $C^*$-algebra,  $\a\in \End(A)$ an injective, unital $*$-endomorphism, $\ca_\infty=\varinjlim \ca$ the inductive limit  described in \eqref{eq:CstarIndLim1}, and $(\cl_\infty,\ch_\infty,D_{\infty};\cam_\infty,\t_\infty)$ a semifinite spectral triple of dimension $p$ on $\ca_\infty$. If the morphism $\a_\infty\in {\rm Aut}(\ca_\infty)$ is Lip-semibounded, then
 there exists a semifinite spectral triple $(\cl_\rtimes,\ch_\rtimes,D_\rtimes;\cam_\rtimes,\t_\rtimes)$ of dimension $p+1$   on the crossed product $C^*$-algebra $\ca\rtimes_\a \bn$.
\end{theorem}
\begin{proof}
Note that $\ca\rtimes_\a \bn=\ca_\infty\rtimes_{\a_\infty} \bn$. Now the claim follows by applying the previous proposition.
\end{proof}

\section[Spectral triples for crossed products]{Spectral triples for crossed products generated by self-coverings}

In this section we exhibit some examples of semifinite spectral triples for crossed products with respect to an endomorphism: the self-covering of a $p$-torus, the self-covering of the rational rotation algebra,  the endomorphism UHF algebra given by the shift, and the self-covering of the Sierpi\'nski gasket. 
In this paper we consider two pictures of the inductive limits. One is what we call the Cuntz picture. The other one deals with an increasing sequence of algebras  $\ca_i$ with the morphisms $\varphi_i: \ca_i\to \ca_{i+1}$ being the inclusions, which entails that the morphisms $\alpha_i: \ca_i\to\ca_i$ are injective.
The following result gives a more detailed description of the second picture. 

\begin{proposition}\label{prop41}
Given a family of algebras $\{\ca_i\}_{i\geq 1}$, a morphism $\a_1: \ca_1\to\ca_1$, a collection of isomorphisms $\b_i: \ca_i\to\ca_{i+1}$ for all $i\in \bn$, one can obtain the following commuting diagram
\begin{equation*} 
	\xymatrix{
	& \ca_1 \ar[rr]^{ \varphi_1 } \ar[dd]^{\a_1} &&  \ca_2 \ar[rr]^{ \varphi_2  } \ar[dd]^{\a_2} && \ca_3 \ar[rr]^{  \varphi_3  } \ar[dd]^{\a_3} && \cdots \ar[r]  &  \ca_\infty \ar[dd]^{\a_\infty} \\
	&&&&&&&&\\
	& \ca_1 \ar[rr]^{  \varphi_1  } \ar[uurr]^{\b_1} && \ca_2 \ar[rr]^{ \varphi_2  } \ar[uurr]^{\b_2} && \ca_3 \ar[rr]^{ \varphi_3  } \ar[uurr]^{\b_3} && \cdots \ar[r] & \ca_\infty
         }
\end{equation*} 
where the morphisms $\a_i: \ca_i\to\ca_i$ are defined by the formula $\a_i :=\b_{i-1}\circ \a_{i-1}\circ \b_{i-1}^{-1}$ for $i\geq 2$, 
$\varphi_1 := \beta_1\circ \a_1$, $\varphi_i := \a_{i+1}\circ \b_i =\b_i\circ \a_i$ for $i\geq 2$.  
Moreover, the morphisms $\{\varphi_i\}_{i\geq 1}$ give rise to an inductive limit that we denote by $\ca_\infty$ and
 the former morphisms $\{\a_i\}_{i\geq 1}$ and $\{\b_i\}_{i\geq 1}$ induce morphisms $\a_\infty, \b_\infty: \ca_\infty\to\ca_\infty$   that are inverses of each other.
\end{proposition}
\begin{proof}
The first part of the statement, namely the one concerning the commuting diagram, follows by direct computations.
Now we take care of the second part concerning the morphisms $\a_\infty$ and $\b_\infty$. We observe that 
\begin{align*} 
\a_\infty(f_1, f_2, \ldots) & = (\a_1(f_1), \a_{2}(f_2),\ldots)\\
\b_\infty(f_1, f_2, \ldots) & = (0,\b_{1}(f_1), \b_{2} (f_2),\ldots)
\end{align*}
for all $(f_1, f_2,\ldots)\in\ca_\infty$.
On the one hand, we have that
\begin{align*} 
\a_\infty\circ \b_\infty(f_1, f_2, \ldots) & = \a_\infty(0,\b_{1}(f_1), \b_{2} (f_2),\ldots)\\
& = (0,\a_2\circ\b_{1}(f_1), \a_3\circ\b_{2} (f_2),\ldots)\; .
\end{align*}
On the other hand, we have that
\begin{align*} 
\b_\infty\circ\a_\infty (f_1, f_2, \ldots)  & = \b_\infty(\a_1(f_1), \a_{2}(f_2),\ldots)\\
& = (0,\b_1\circ \a_1(f_1), \b_2\circ\a_{2}(f_2),\ldots)\; .
\end{align*}
Since $ \a_{i+1}\circ \b_i =\b_i\circ \a_i$ we are done.
\end{proof}

\begin{notation}\label{notazioneBs}
Before the discussion of the   examples, we introduce some notation. We will consider an invertible matrix $B\in M_p(\bz)$ and we will set $A:=(B^T)^{-1}$. The following exact sequence will play a role in the definition of some of the Dirac operators
$$
0\to \bz^p\to A\bz^p\to \widehat{\bz_B}:= A\bz^p/\bz^p\to 0\; .
$$ 
Moreover, we will consider a section $s: \widehat{\bz_B}\to A\bz^p$ such that $s(\cdot)\in [0,1)^p$. We set $s_h(x):=A^{h-1}s(x)$ as in \cite{AiGuIs01}, p. 1387-1388. Note that $|\widehat{\bz_B}|=|\det(B)|=:r$.
\end{notation}

\medskip

\subsection{The crossed product for the self-coverings of the $p$-torus}

We begin with the case of tori. 
The $p$-torus $\bt^p :=\mathbb{R}^p/\mathbb{Z}^p$ can be endowed with a Dirac operator acting on the Hilbert space $\ch_0 := \bc^{2^{[p/2]}} \otimes L^2(\bt^p,dm)$ 
$$
D_0 :=-i\sum_{a=1}^p \eps_a \otimes \partial^a,
$$
where the matrices $\eps_a= (\eps_a)^*\in M_{2^{[p/2]}}(\bc)$, $\eps_a \eps_b + \eps_b \eps_a=2 \delta_{a,b}$, furnish a representation of the Clifford algebra for the $p$-torus (see \cite{Spin} for more information on Dirac operators).
Then, we may consider the following spectral triple 
$$
(\cl_0 :=C^1(\bt^p),  \ch_0, D_0).
$$
We recall that the spectral triple considered for the torus is even precisely when $p$ is even.   

With the above notation and $B\in M_p(\bz)$, let $\pi:t\in\bt^p\mapsto Bt\in\bt^p$ be  the self-covering, $\a(f)(t)=f(Bt)$ the associated endomorphism of $\ca =C(\bt^p)$.  
Then we consider the inductive system \eqref{eq:CstarIndLim1} and construct the inductive limit $\ca_\infty=\displaystyle\varinjlim\ca_n$. 
An alternative description is given by the following isomorphic inductive family: $\ca_n$ consists of continuous $B^n\bz^p$-periodic functions on $\br^p$, and the embedding is the inclusion.  
In the following we denote by $\bt_n$ the $p$-torus $\br^p/B^n\bz^p$.

Assume now that $B$ is purely expanding, namely $\|B^n v\|$ goes to infinity
 for all vectors $v\neq 0$, hence $\|A\|<1$, where
$A=(B^T)^{-1}$.
In \cite{AiGuIs01}, we produced a semifinite spectral triple on $\ca_\infty=\varinjlim C(\bt_n)$. More precisely,  we constructed a Dirac operator\footnote{The symbol $s_h(\cdot)$ denotes the section defined in Notation \ref{notazioneBs}.} $D_\infty$ acting on $\ch_\infty:= \bc^{2^{[p/2]}} \otimes L^2(\bt^p,dm)\otimes L^2(\car,\tau)$
$$
D_\infty := D_0 \otimes I - 2\pi \sum_{a=1}^p \eps_a \otimes I \otimes \bigg(  \sum_{h=1}^\infty I^{\otimes h-1} \otimes \diag(s_{h}(\cdot)^a)  \bigg),
$$
the algebra $\cl_\infty := \cup_{n\in\bn} C^1(\bt_n)\subset \ca_\infty$ embeds into the injective limit 
$$
\varinjlim \cb(\ch_0)\otimes M_{{\molt}^n}(\bc) = \cb(\bc^{2^{[p/2]}} \otimes L^2(\bt^p,dm)) \otimes \mathrm{UHF}_r\; ,
$$
where $\mathrm{UHF}_r$ denotes the infinite tensor product of $M_r(\bc)$, see Section 4.3 for more details. The C$^*$-algebra $\cb(\bc^{2^{[p/2]}} \otimes L^2(\bt^p,dm)) \otimes \mathrm{UHF}_r$ in turn embeds into $\cam_\infty := \cb(\bc^{2^{[p/2]}} \otimes L^2(\bt^p,dm)) \otimes \car$, where $\car$ denotes the unique injective type II$_1$ factor obtained as the weak closure of the UHF algebra in the GNS representation of the unital trace, and we denote by $\t_\infty := Tr\otimes\t_\car$ the trace on $\cam_\infty$. Then
$( \cl_\infty, \ch_\infty,D_\infty; \cam_\infty, \t_\infty)$ is a finitely summable, semifinite, spectral triple on $\varinjlim \ca_n$, with Hausdorff dimension $p$.


\begin{theorem}\label{teo-p-toro}
Under the above hypotheses and with the notation of the former section, $C(\bt^p)\rtimes_\a\bn$ can be endowed with the finitely summable semifinite spectral triple $(\cl_\rtimes, \ch_\rtimes,D_\rtimes; \cam_\rtimes, \t_\rtimes)$ of Theorem \ref{triple-cross-prod-N}, with Hausdorff dimension $p+1$. 
\end{theorem}
\begin{proof} 
In order to construct a spectral triple on $C(\bt^p)\rtimes_\alpha \mathbb{N}$, according to Theorem \ref{triple-cross-prod-N}, we only need to check that $\alpha_\infty$ is Lip-semibounded, that is
\begin{displaymath}
	\sup \{\Arrowvert [D_\infty ,\alpha_\infty^{-n}(f)]\Arrowvert, n\in\bn\}<\infty, \quad \forall f\in \cl_\infty = \cup_{n\in\bn} C^1(\bt_n).
\end{displaymath}
Let $f\in C^1(\bt_k)$. 
As observed in \cite{AiGuIs01}, the seminorms $L_{D_\infty}$, $L_{D_1}$, $L_{D_2}$, \ldots \; are compatible and we have that 
$$
\| [D_\infty ,\alpha_\infty^{-n}(f)] \| = \| [D_0,f\circ B^{-n}] \|
$$
Moreover, by using the relation $\eps_a \eps_b + \eps_b \eps_a=2 \delta_{a,b}$ we obtain the following equalities
\begin{align*}
 \| [D_0,f] \|^2 & =  \left\| \sum_{a=1}^p \eps_a \otimes \partial^a(f) \right\|^2\\
 & =  \left\| \left(\sum_{a=1}^p \eps_a \otimes \partial^a(f)\right)^* \left(\sum_{a=1}^p \eps_a \otimes \partial^a(f)\right)\right\| \\
 & =  \left\| \sum_{a=1}^p (\eps_a)^2 \otimes |\partial^a f|^2\right\| =  \left\| \sum_{a=1}^p 1 \otimes |\partial^a f|^2\right\|
\end{align*}
Now we compute $\| [D_0,f\circ B^{-n}] \|$. Setting  $X=B^{-n}$ for simplicity, we have that
\begin{align*}
\| [D_0,f\circ X] \|^2 & =  \left\| \sum_{a=1}^p \eps_a \otimes \partial^a(f\circ X) \right\|^2 \\
& =  \left\| \left(\sum_{a=1}^p \eps_a \otimes \partial^a(f\circ X)\right)^* \left(\sum_{a=1}^p \eps_a \otimes \partial^a(f\circ X)\right)\right\| \\
& =  \left\| \left(\sum_{a=1}^p \eps_a \otimes \left( \sum_{i=1}^p X_{a,i} (\partial^i \ov{f})\circ X \right) \right) \left(\sum_{b=1}^p \eps_b \otimes  \left( \sum_{j=1}^p X_{b,j}(\partial^j f)\circ X \right)  \right)\right\| \\
& =  \left\| \sum_{a=1}^p (\eps_a)^2 \otimes  \sum_{i,j=1}^p X_{a,i} X_{a,j} (\partial^i \ov{f})\circ X \cdot (\partial^j f)\circ X \right.\\
& \quad + \sum_{a<b} \eps_a \eps_b  \otimes  \sum_{i,j=1}^p X_{a,i} X_{b,j} (\partial^i \ov{f})\circ X \cdot (\partial^j f)\circ X  \\
& \quad  \left. + \sum_{a>b} \eps_a \eps_b  \otimes  \sum_{i,j=1}^p X_{a,i} X_{b,j} (\partial^i \ov{f})\circ X \cdot (\partial^j f)\circ X \right\| \\
& =  \left\| \sum_{a=1}^p 1 \otimes \left(   \sum_{i,j=1}^p X_{a,i} X_{a,j} (\partial^i \ov{f})\circ X \cdot (\partial^j f)\circ X  \right) \right\| \\
& = \| \big( (\nabla f)\circ X, X^* X (\nabla f)\circ X \big)\| \\
& \leq \| X^* X  \|  \left\| \sum_{a=1}^p 1 \otimes (\partial^a f)^2\right\| =   \| X\|^2 \| [D,f] \|^2 \; .
\end{align*}
These computations and the hypothesis on $B$ being purely expanding (cf. Proposition 2.6 in \cite{AiGuIs01}) imply that
\begin{displaymath}
	\sup\{\Arrowvert [D_\infty ,\alpha_\infty^{-n}(f)]\Arrowvert, n\in\bn\}\leq\sup\{ \Arrowvert B^{-n}\Arrowvert \Arrowvert [D_\infty,f]\Arrowvert, n\in\bn\}<\infty\; .
\end{displaymath}
\end{proof}

\medskip

\subsection{The crossed product for the self-coverings of the  rational rotation algebra}

The present  example is associated with a regular noncommutative self-covering with finite abelian group of deck transformations \cite{AiGuIs01}. 

\begin{definition} \label{def-reg-cov}
A finite (noncommutative) covering with abelian group is an inclusion of  (unital) $C^*$-algebras $\ca\subset \cb$ together with an action of a finite abelian group $\Gamma$ on $\cb$ such that $\ca=\cb^\Gamma$. We will say that $\cb$ is a covering of $\ca$ with deck transformations given by the group $\Gamma$.
\end{definition}


We are now going to give a description of the rational rotation algebra making small modifications to the description of $A_\theta$, $\theta=p/q\in\mathbb{Q}$, seen in \cite{BEEK}. We observe that $A_\theta$ reduces to $C(\bt^2)$ in the case $\theta\in\bz$.
Consider the following matrices 
\begin{eqnarray*}
  (U_0)_{hk} = \delta_{h,k}e^{2\pi i(k-1) \theta }, \quad
  (V_0)_{hk} = \delta_{h+1,k} +\delta_{h,q}\delta_{k,1} \in M_q(\mathbb{C})
\end{eqnarray*}
and $W_0(n) := U_0^{n_1}V_0^{n_2}$, for all $n=(n_1,n_2)\in\bz^2$. Let $p',p''\in\bn$, $p',p''<q$, be such that $pp'+1=n'q$, $pp''-1=n''q$, for some $n',n''\in\bn$, and introduce 
$P := \begin{pmatrix}
0 & p' \\
p'' & 0
\end{pmatrix}
$, and  
$$
\widetilde{\g}_n(f)(t):=\ad(W_0(P n))[f(t+n)]=V_0^{-p''n_1}U_0^{-p'n_2}f(t+n)U_0^{p'n_2}V_0^{p''n_1},
$$ 
for all $t\in\br^2$, $n\in\bz^2$. We have the following description of $A_\theta$  (cf. \cite{BEEK})
$$
A_\theta=\{f\in C(\mathbb{R}^2, M_q(\mathbb{C})) \, : \, f  = \widetilde{\g}_{n}(f),  n\in\bz^2 \}.
$$
This algebra comes with a natural trace 
$$
\tau(f):= \frac{1}{q}\int_{\bt_0} \tr(f(t))dt, 
$$
where we are considering the Haar measure on $\bt_0:=\br^2/B\bz^2$ and $\tr(A)=\sum_i a_{ii}$. We observe that the function $\tr(f(t))$ is $\bz^2$-periodic.

Define 
\begin{align*}
U(t_1,t_2)&:=e^{-2\pi i t_1/q} U_0\\
V(t_1,t_2)&:=e^{-2\pi i t_2/q} V_0
\end{align*}
and
\begin{displaymath}
		\cl_\theta :=\left\{\sum_{r,s}a_{rs}U^rV^s :  (a_{rs})\in S(\mathbb{Z}^2) 	\right\},
\end{displaymath}
where $S(\mathbb{Z}^2)$ is  the set of rapidly decreasing sequences. It is clear that the derivations $\partial_1$ and $\partial_2$, defined as follows on the generators, extend to $\cl_\theta$
\begin{eqnarray*}
	\partial_1(U^hV^k)&=&2\pi ihU^hV^k\\
	\partial_2(U^hV^k)&=&2\pi ikU^hV^k.
\end{eqnarray*} 
Moreover, the above derivations extend to densely defined derivations both on $A_\theta$ and $L^2(A_\theta,\tau)$.

We still denote these extensions with the same symbols. 
We may consider the following spectral triple (see \cite{GBFV})
 \begin{eqnarray*}
	(\cl_0:=\cl_\theta ,  \ch_0:=\bc^{2} \otimes L^2(A_\theta,\tau), D_0 :=-i(\eps_1 \otimes \partial_1+\eps_2 \otimes \partial_2)),
\end{eqnarray*}
where $\eps_1, \eps_2$ denote the Pauli matrices. 

Given the integer-valued matrix $B\in M_2(\bz)$ such that $\det(B)\equiv_q 1$, there is an associated endomorphism  $\alpha: A_\theta\to A_\theta$ defined by $\a(f)(t)=f(Bt)$, \cite{Stacey}. Then, we consider the inductive limit $\ca_\infty=\displaystyle\varinjlim\ca_n$ as in  \eqref{eq:CstarIndLim1}.  
As in the case of the torus one can consider the following isomorphic inductive family: $\ca_n$ consists of continuous $B^n\bz^2$-invariant  matrix-valued functions on $\br^2$, i.e
$$
\ca_n:=\{f\in C(\mathbb{R}^2, M_q(\mathbb{C})) \, : \, f =\widetilde{\g}_{B^n k}(f),  k\in\bz^2 \},
$$ 
with trace
$$
\tau_n(f)=\frac{1}{q |\!\det B^n|}\int_{\bt_n}\tr(f(t))dt,
$$
and the embedding is unital inclusion $\alpha_{n+1,n}: \ca_n\hookrightarrow \ca_{n+1}$.  In particular, $\ca_0=\ca$, and $\ca_1=\cb$. This means that $\ca_\infty$ may be considered as a solenoid $C^*$-algebra  (cf. \cite{McCord}, \cite{LP2}).

On the $n$-th noncommutative covering $\ca_n$, the formula of the Dirac operator doesn't change and we can consider the following spectral triple
\begin{eqnarray*}
	(\cl_\theta^{(n)} ,  \bc^{2} \otimes L^2(\ca_n,\tau), D=-i(\eps_1 \otimes \partial_1 + \eps_2 \otimes \partial_2)).
\end{eqnarray*}

In \cite{AiGuIs01}, we produced a semifinite spectral triple on $\ca_\infty=\varinjlim \ca_n$. More precisely, we constructed a Dirac operator $D_\infty$ acting on $\ch_\infty:= \bc^{2} \otimes L^2(\ca_0,\tau_0)\otimes L^2(\car,\tau)$
\[
D_\infty := D_0 \otimes I - 2\pi \sum_{a=1}^2 \eps_a \otimes I \otimes \bigg(  \sum_{h=1}^\infty I^{\otimes h-1} \otimes \diag(s_{h}(\cdot)^a)  \bigg),
\]
the algebra $\ca_\infty$ embeds into the injective limit 
$$
\varinjlim \cb(\bc^{2} \otimes L^2(\ca_0,\tau_0))\otimes M_{{\molt}^n}(\bc) = \cb(\bc^{2} \otimes L^2(\ca_0,\tau_0)) \otimes \mathrm{UHF}_{\molt}
$$
which in turn embeds into $\cam_\infty := \cb(\bc^{2} \otimes L^2(\ca_0,\tau_0)) \otimes \car$, which is endowed with the trace $\t_\infty := Tr\otimes\t_\car$. Then
$(\cl_\infty, \ch_\infty,D_\infty; \cam_\infty, \t_\infty)$ is a finitely summable, semifinite, spectral triple on $\varinjlim \ca_n$, with Hausdorff dimension $2$ (\cite{AiGuIs01}, Theorem 3.7).


\begin{theorem}
Under the above hypotheses and with the notation of the former section, $A_\theta\rtimes_\a\bn$ can be endowed with the finitely summable semifinite spectral triple $(\cl_\rtimes, \ch_\rtimes,D_\rtimes; \cam_\rtimes, \t_\rtimes)$ of Theorem \ref{triple-cross-prod-N}, with Hausdorff dimension $3$. 
\end{theorem}
\begin{proof}
According to  Theorem \ref{triple-cross-prod-N} we only need to check that $\alpha_\infty$ is Lip-semibounded, that is
$$
	\sup\{\Arrowvert [D_\infty,\alpha_\infty^{-n}(f)]\Arrowvert, n\in\bn\} < \infty, \quad \forall f\in \cl_\infty \; .
$$
This is true because similar computations to those in the proof of Theorem \ref{teo-p-toro} yield 
\begin{align*}
	\sup\{\Arrowvert [D_\infty,\alpha_\infty^{-n}(f)]\Arrowvert, n\in\bn\}
	\leq	\sup\{ \| B^{-n}\| \Arrowvert [D_\infty,f]\|, n\in\bn\}
\end{align*}
The hypothesis of $B$ being purely expanding ensures that $\sup\{\Arrowvert [D_\infty,\alpha_\infty^{-n}(f)]\Arrowvert, n\in\bn\}$ is finite.
\end{proof}

\medskip

\subsection{The crossed product for the shift-endomorphism of the UHF-algebra}\label{UHF}

Consider now the case of the UHF-algebra. 
This algebra is defined as the inductive limit of the following  sequence of finite dimensional matrix algebras:   
\begin{eqnarray*}
	M_0 & = & M_\molt(\mathbb{C})\\
	M_n & = & M_{n-1}\otimes M_\molt(\mathbb{C}) \quad n\geq 1,
\end{eqnarray*}
with maps $\phi_{ij}: M_j\to M_i$ given by $\phi_{ij}(a_j)=a_j\otimes 1$.
We denote by $\ca$ the $C^*$-algebra UHF$_r$  and set $M_{-1}=\mathbb{C}1_\ca$  in the inductive limit defining the above algebra. The $C^*$-algebra $\ca$ has a unique normalized trace that we denote by $\tau$.

Consider the projection $P_n:L^2(\ca,\tau)\to L^2(M_n,{\rm Tr})$, where ${\rm Tr}: M_n \to \mathbb{C}$ is the normalized trace, and define 
\begin{eqnarray*}
	Q_n&:=& P_n-P_{n-1}, \quad n\geq 0,\\
	E(x)&:=&\tau(x)1_\ca\,.
\end{eqnarray*} 
For any $s>1$, Christensen and Ivan \cite{Chris} defined the following spectral triple for the algebra UHF$_\molt$ 
\begin{eqnarray*} 
	(\cl_0, L^2(\ca,\tau),D_0=\sum_{n\geq 0} \molt^{ns}Q_n )
\end{eqnarray*}
where $\cl_0$ is the algebra consisting of the elements of $\ca$ with bounded commutator with $D_0$. It was proved that for any such value of the parameter $s$, this spectral triple induces a metric which defines a topology equivalent to the weak$^*$-topology on the state space (\cite[Theorem 3.1]{Chris}).

We consider the endomorphism of $\ca$ given by the right shift, $\a(x)=1\otimes x$. Then as in  \eqref{eq:CstarIndLim1} we may consider the inductive limit $\ca_\infty=\displaystyle\varinjlim\ca_n$. As in the previous sections, we have the following isomorphic inductive family: $\ca_i$  is defined as
\begin{eqnarray*}
		\ca_0&=& \ca;\\
		\ca_n &=& M_\molt(\mathbb{C})^{\otimes n} \otimes \ca_0;\\
		\ca_\infty &=& \varinjlim \ca_i
\end{eqnarray*}
and the embedding is the inclusion.
It is easy to see that $\ca_\infty$ is again the UHF-algebra of the same type, since the corresponding supernatural number is the same.

In \cite{AiGuIs01}, we produced a semifinite spectral triple on $\varinjlim \ca_n$. More precisely,  we defined the following Dirac operator acting on $\ch_\infty := L^2(\car, \tau)\otimes L^2(\ca_0, \tau)$
\begin{equation}
D_{\infty}=I_{-\infty,-1} \otimes D_0+\sum_{k=1}^\infty \molt^{-sk} I_{-\infty,-k-1}\otimes F\otimes E,
\end{equation}
where $I_{-\infty,k}$ is the identity on the factors with indices in $[-\infty, k]$, $F: M_r(\bc)\to M_r(\bc)$ is defined as $F(x):=x-\tr(x)1$ for $x\in M_r(\bc)$, and the algebra $\ca_\infty$ embeds in the injective limit 
$$
\varinjlim \cb(L^2(\ca_0,\tau)) \otimes M_{{\molt}^n}(\bc)=\cb(L^2(\ca_0,\tau))\otimes  \mathrm{UHF}_r
$$
Set $\cl_\infty = \cup_n\cl_n$, $\cam_\infty = \car \otimes \cb(L^2(\ca_{0},\tau))$, $\t_\infty:=\t_\car\otimes Tr$. 
Then $(\cl_\infty,\ch_\infty,D_{\infty}; \cam_\infty,\t_\infty)$ is a finitely summable, semifinite, spectral triple, with Hausdorff dimension $2/s$ (\cite{AiGuIs01}, Theorem 5.6).


\begin{theorem}\label{UHFcrossedprod}
Under the above hypotheses and the notation of the former section, ${\rm UHF}_r\rtimes_\alpha \bn$ can be endowed with the finitely summable semifinite spectral triple $(\cl_\rtimes, \ch_\rtimes,D_\rtimes; \cam_\rtimes,\t_\rtimes)$ of Theorem \ref{triple-cross-prod-N}, with Hausdorff dimension $1+2/s$. 
\end{theorem}	
\begin{proof}
According to Theorem \ref{triple-cross-prod-N}, in order to construct a spectral triple on $\ca \rtimes_\a\bn$  we only need to check that $\alpha_\infty$ is Lip-semibounded, that is
\begin{displaymath}
	\sup\{\Arrowvert [D_\infty,\alpha_\infty^{-k}(f)]\Arrowvert, k\in\bn\}< \infty, \quad \forall f\in \cl_\infty.
\end{displaymath}
This is true because
\begin{align*}
	\Arrowvert [D_\infty,\alpha_\infty^{-k}(f)]\Arrowvert = \molt^{-ks}\Arrowvert [D_\infty,f]\Arrowvert.
\end{align*}
In fact, let $f=(\bigotimes_{k=-\infty}^{-n-1}I)\otimes a\in \ca_n$, 
$\alpha_\infty^{k}(f)=(\bigotimes_{j=-\infty}^{-n+k-1}I)\otimes a\in \ca_{n-k}$ for $k\in\bz$. 

The Hilbert space on which $D_\infty$ acts is the completion of $\ca_\infty$. 
On this Hilbert space, we consider the right shift on the factors  and we denote it by $U_\alpha$. We set $\Phi := \ad (U_\alpha)$. Then we have that
\begin{eqnarray*}
	 [D_\infty,\alpha_\infty^{-k}(f)] &=&  \sum_{h\in\bz} \molt^{hs}[Q_h,\left(\bigotimes_{j=-\infty}^{-n-k-1}I\right)\otimes a]\\
	&=& \Phi^{-k} \left(\sum_{h\in\bz} \molt^{hs}[Q_{k+h},\left(\bigotimes_{j=-\infty}^{n-1}I\right)\otimes a]\right)\\
	&=& \molt^{-ks}   \Phi^{-k}([D_\infty,f]) 
\end{eqnarray*}
where we used that $\Phi(Q_h)=Q_{h+1}$ and $\Phi\upharpoonright_{\ca_\infty} = \alpha_\infty$. 
\end{proof}
In the theorem above, we considered the $C^*$-algebra ${\rm UHF}_r\rtimes_\alpha \bn$. We note that the crossed product of the UHF of type $2^\infty$ under the action of the bilateral shift, namely the $C^*$-algebra ${\rm UHF}_2 \rtimes_\alpha  \bz$, 
  was studied in \cite{BKRS}.
 
\medskip

\subsection{The crossed product for the self-coverings of the Sierpi\'nski gasket}

We conclude this paper with the case of a self-covering of the Sierpi\'nski gasket that was studied by the authors in \cite{AGI3}.
The Sierpi\'nski gasket is the self-similar fractal determined by $3$ similarities with scaling parameter 1/2 centered in the vertices 
$v_0=(0,0)$, $v_1=(1/2,\sqrt{3}/2)$, $v_2=(1,0)$, 
namely 
the non-empty, compact set $K$, such that
$$
K=\bigcup_{j=0,1,2}w_j(K),
$$
where $w_j$ is the dilation around $v_j$ with contraction parameter $1/2$ (see Figure \ref{fig:covering}). 
Denote by $V_0(K)$ the set  $\{v_0, v_1, v_2\}$, and let $E_0(K):=\{ (p,q) : p,q\in V_0, p\neq q\}$.  
We call an element of the family $\{w_{i_1} \circ \dots \circ w_{i_k}(K):k\geq0\}$ a {\it cell}, and call its diameter the size of the cell.
We call an element of the family $E(K)=\{w_{i_1} \circ \dots \circ w_{i_k}(e):k\geq0, e\in E_{0}(K)\}$  an {\it (oriented) edge} of $K$ and we denote by $e^-$ (resp. $e^+$) the source (resp. the target) of the oriented edge $e$.
Note that a cell $C:=w_{i_1} \circ \dots \circ w_{i_k}(K)$ has $\size(C)=2^{-k}$ and, if $e_0\in E_{0}(K)$, then $e=w_{i_1} \circ \dots \circ w_{i_k}(e_0)$ has length $2^{-k}$.

In the following we shall consider $K_0:=K$, $E_0:=E(K)$, $K_n:=w_0^{-n}K_0$.
Let us now consider  the middle point $x_{i,i+1}$ of the segment $(w_0^{-1}v_i,w_0^{-1}v_{i+1})$, $i=0,1,2$, the map $R_{i+1,i}:w_0^{-1}w_{i}K\to w_0^{-1}w_{i+1}K$ consisting of the rotation of $\frac43\pi$ around the point $x_{i,i+1}$, $i=0,1,2$. 

We then construct the coverings $p:K_1\to K$ and $\phi: K\to K$ given by
$$
p(x)=
\begin{cases}
x,&x\in K,\\
R_{0,1}(x),&x\in w_0^{-1}w_1 K,\\
R_{0,2}(x),&x\in w_0^{-1}w_2 K,
\end{cases}
$$
and
\begin{displaymath}
\phi(x)=\left\{
 	\begin{array}{ll}
 		w_0^{-1}x & \textrm{if $x\in C_0$} \\ 
 		R_{0,1}(w_0^{-1}(x)) & \textrm{if $x\in C_1$}\\
 		R_{0,2}(w_0^{-1}(x)) & \textrm{if $x\in C_2$}\\
 	\end{array}
 		\right.\\
\end{displaymath}
Note that $p(x)=\phi(w_0(x))$ for all $x\in K_1$ (see Figure \ref{fig:covering}).
\begin{figure} 
\scalebox{1.75}{\def\trianglewidth{2cm}%
\pgfdeclarelindenmayersystem{Sierpinski triangle}{
    \symbol{X}{\pgflsystemdrawforward}
    \symbol{Y}{\pgflsystemdrawforward}
    \rule{X -> X-Y+X+Y-X}
    \rule{Y -> YY}
}
\foreach \level in {6}{

\tikzset{
    l-system={step=\trianglewidth/(2^\level), order=\level, angle=-120}
}%
\begin{tikzpicture}
    \fill [black] (0,0) -- ++(0:\trianglewidth) -- ++(120:\trianglewidth) -- cycle;
    \draw [draw=none] (0,0) l-system
    [l-system={Sierpinski triangle, axiom=X},fill=white];

\node (bbb) at (-.05,-0.075) {$\scalebox{.5}{$v_0$}$}; 
\node (bbb) at (1,1.85) {$\scalebox{.5}{$v_2$}$}; 
\node (bbb) at (2.09,-.075) {$\scalebox{.5}{$v_1$}$};  
\node (bbb) at (0,-.87) {$\scalebox{.5}{$\;$}$};  

\end{tikzpicture}
}%
}
\scalebox{3}{
\def\trianglewidth{2cm}%
\foreach \level in {6}{

\tikzset{
    l-system={step=\trianglewidth/(2^\level), order=\level, angle=-120}
}%
\begin{tikzpicture}
    \fill [black] (0,0) -- ++(0:\trianglewidth) -- ++(120:\trianglewidth) -- cycle;
    \draw [draw=none] (0,0) l-system
    [l-system={Sierpinski triangle, axiom=X},fill=white];
 
\draw[line width=0.05mm, <-] (.75,-.1) to[out=-90,in=-90] (1.25,-.1);
\draw[line width=0.05mm, <-] (0.25,.6) to[out=135,in=120]  (.5,1);

\node (bbb) at (-.05,-0.075) {$\scalebox{.25}{$v_0$}$}; 
\node (bbb) at (.395,.87) {$\scalebox{.25}{$x_{2,0}$}$}; 
\node (bbb) at (1.65,.87) {$\scalebox{.25}{$x_{1,2}$}$}; 
\node (bbb) at (1,1.85) {$\scalebox{.25}{$w_0^{-1}v_2$}$}; 
\node (bbb) at (1,-.075) {$\scalebox{.25}{$x_{0,1}=v_1$}$}; 
\node (bbb) at (2.07,-.075) {$\scalebox{.25}{$w_0^{-1}v_1$}$}; 
\node (bbb) at (0.1,.89) {$\scalebox{.25}{$R_{0,2}$}$}; 
\node (bbb) at (1.1,-.35) {$\scalebox{.25}{$R_{0,1}$}$}; 

\end{tikzpicture}
}
}
\label{fig:covering}
\caption{The Sierpi\'nski gasket $K=K_0$ and the covering map $p_1=p: K_1\to K$.}
\end{figure}
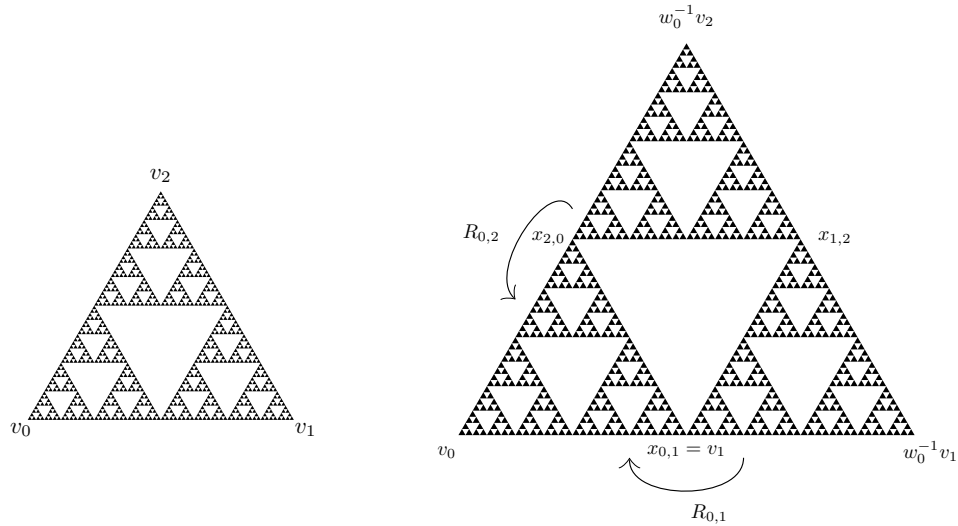

Similarly, for every $n\geq 0$, we define a family of coverings $p_n: K_{n+1}\to K_{n}$ and $\phi_n: K_n\to K_n$ by $p_{n+1} := w_0^{-n} \circ p \circ w_0^{n}$ and $\phi_n := w_0^{-n} \circ \phi \circ w_0^{n}$.

\begin{proposition} 
The following diagrams are commutative
\begin{equation*} 
	\xymatrix{
	& K_0  &&  K_1 \ar[ll]^{ p_1  }  && K_2 \ar[ll]^{ p_2 }  && \cdots \ar[ll]^{p_3} \\
	&&&&&&&&\\
	& K_0 \ar[uu]^{  \phi_0  }  && K_1 \ar[ll]^{ p_1  } \ar[uu]^{ \phi_1 } && K_2 \ar[ll]^{ p_2  } \ar[uu]^{ \phi_2 } && \cdots \ar[ll]^{ p_3 }
         }
\end{equation*}
\end{proposition}
\begin{proof}
Indeed, first note that $\phi_0 \circ p_1=\phi \circ p=p\circ \phi_1$ and $w_0\circ \phi_1=p_1$, which implies that $p_1\circ\phi_1=\phi_0\circ p_1$.
Then, for any $n\geq 1$ we have 
\begin{align*}
p_n \circ \phi_n & = w_0^{-n+1} \circ p \circ w_0^{n-1} \circ w_0^{-n} \circ \phi \circ w_0^n= w_0^{-n+1} \circ p \circ w_0^{-1} \circ \phi \circ w_0 \circ w_0^{n-1} \\
& = w_0^{-n+1} \circ p_1 \circ \phi_1 \circ w_0^{n-1} = w_0^{-n+1} \circ \phi_0 \circ p_1 \circ w_0^{n-1} = \phi_{n-1} \circ p_n.
\end{align*}
\end{proof}


It follows that the maps $\{\phi_n\}_{n\geq 0}$ induce a map in the projective limit and by functoriality a map on $\varinjlim C(K_i)$ which we denote by $\alpha_\infty$. 
An element $f\in C(K_n)$ can be seen in $\varinjlim C(K_i)$ as the sequence $[f]=(0_{n}, f, f\circ p_{n+1}, f\circ p_{n, n+2}, \ldots )$, where  $p_{n, n+k}:= p_{n+1}\circ \cdots\circ p_{n+k}$.
Accordingly the map $\alpha_\infty$ reads as
$$
\alpha_\infty[f] := (0_{n}, f\circ \phi_n, f\circ p_{n+1} \circ \phi_{n+1}, f\circ p_{n, n+2} \circ \phi_{n+2}, \ldots ).
$$
By functoriality each $(\phi_n)^*: C(K_n)\to C(K_n)$ is a proper endomorphism, that is, it is injective, but not surjective. 
With the notation of Proposition \ref{prop41}, we set $\b_i$ equal to $w_0^*$ for all $i\geq 0$.
Thanks to Proposition 3.1, the map $\alpha_\infty$ is invertible and its inverse is given by 
$$
\alpha_\infty^{-1}[f] := (0_{n+1}, f\circ w_0, f\circ p_{n+1} \circ w_0, f\circ p_{n, n+2} \circ w_0, \ldots ).
$$
Denote by $E_n:=\{w_0^{-n}e, e\in E(K)\}$, $E_\infty:=\cup_{n\geq 0} E_n$, $E^n:=\{e\in E_\infty, \length(e)=2^n\}$, $P^n$ the projection of $\ell_2(E_\infty)$ onto $\ell_2(E^n)$.
It was shown in \cite[Sec. 6]{AGI3} that $\ca_\infty:=\varinjlim C(K_n)$ supports a semifinite spectral triple $(\cl_\infty,\ch_\infty, D_\infty; \cam_\infty, \t_\infty)$,
where $\cam_\infty:=\pi_\t(B_\infty)''$ is a suitable closure of the geometric operators (see \cite[Sec. 5]{AGI3} for a precise definition),
$D_\infty:=F|D|: \ell^2(E_\infty)\to  \ell^2(E_\infty)$,  
$F$ is the orientation reversing operator on edges and 
$$
|D_\infty|:=\sum_{n\in\bz}2^{-n}P^n. 
$$


\begin{theorem}\label{teo-gasket}
Under the above hypotheses and with the notation of the former section, $C(K)\rtimes_\a\bn$ can be endowed with the finitely summable semifinite spectral triple $(\cl_\rtimes, \ch_\rtimes,D_\rtimes; \cam_\rtimes, \t_\rtimes)$ of Theorem \ref{triple-cross-prod-N}, with Hausdorff dimension $\log_23+1$. 
\end{theorem}
\begin{proof}
According to Theorem \ref{triple-cross-prod-N}, in order to construct a spectral triple on $C(K)\rtimes_\alpha \mathbb{N}$ we only need to check that $\alpha_\infty$ is Lip-semibounded, that is
$$
\sup_{k\geq 0} \Arrowvert[D_\infty, \alpha_\infty^{-k}(f)]\Arrowvert <\infty, \quad \forall f\in \cl_\infty := \cup_{n\geq 0}\ \mathrm{Lip}(K_n).
$$
We are going to show that for any $f\in C(K_n)$ it holds that
$$
\Arrowvert [D_\infty, \alpha_\infty^{-k}(f)] \Arrowvert =\frac{\Arrowvert [D_\infty, f]\Arrowvert}{2^k} \quad k\in \mathbb{N}.
$$
Indeed, since both $p_n$ and $\phi_n$ are isometries, we have that
\begin{eqnarray*}
\Arrowvert [D_\infty, \alpha_\infty^{-k}(f)] \Arrowvert &=& \left\Arrowvert \oplus_{e\in E_\infty} \frac{\alpha_\infty^{-k}(f)(e^+)-\alpha_\infty^{-k}(f)(e^-)}{l(e)} F \right\Arrowvert \\
&=& \left\Arrowvert \oplus_{e\in E_\infty} \frac{f(w_0^k(e^+))-f(w_0^k(e^-))}{l(e)} F \right\Arrowvert \\
&=& \left\Arrowvert \oplus_{e\in E_\infty} \frac{f(w_0^k(e^+))-f(w_0^k(e^-))}{2^k l(w_0^k(e))} F \right\Arrowvert \\
&=& \left\Arrowvert \oplus_{e'\in E_\infty} \frac{f(e'^+)-f(e'^-)}{2^k l(e')} F\right\Arrowvert \\
&=& \frac{\Arrowvert [D_\infty, f]\Arrowvert}{2^k}.
\end{eqnarray*}
\end{proof}

\section*{Acknowledgement}
We thank the referee for the attentive reading of this article and for useful suggestions.
This work was supported by the following institutions:
the ERC Advanced Grant 669240 QUEST "Quantum Algebraic Structures and Models", 
the MIUR PRIN ``Operator Algebras, Noncommutative Geometry and Applications'', 
the INdAM-CNRS GREFI GENCO, and the INdAM GNAMPA. 
V. A. acknowledges the support by the Swiss National Science foundation through the SNF project no. 178756 (Fibred links, L-space covers and algorithmic knot theory).
D. G. and T. I. acknowledge the MIUR Excellence Department Project awarded to the Department of Mathematics, University of Rome Tor Vergata, CUP E83C18000100006.


\end{document}